\documentclass[11pt,a4paper]{amsart}

\usepackage{times}
\usepackage{hyperref}
\hypersetup{
  colorlinks   = true, 
  urlcolor     = blue, 
  linkcolor    = blue, 
  citecolor   = blue 
}
\usepackage{amsmath}
\usepackage{amsfonts}
\usepackage{amssymb}
\usepackage{amsthm}

\newtheorem{thm}{Theorem}[section]

\newtheorem{thmx}{Theorem}

\newenvironment{thmp}{\stepcounter{thm}\begin{thmx} }{\end{thmx}}

\newtheorem{cor}[thm]{Corollary}
\newtheorem{lem}[thm]{Lemma}
\newtheorem{prop}[thm]{Proposition}
\newtheorem{teo}[thm]{Theorem}

\theoremstyle{definition}
\newtheorem{defi}[thm]{Definition}

\newtheorem{rem}[thm]{Remark}

\DeclareMathOperator{\Iso}{Iso}

\DeclareMathOperator{\Geo}{Geo}
\DeclareMathOperator{\OptGeo}{OptGeo}

\DeclareMathOperator{\Lip}{Lip}
\DeclareMathOperator{\LIP}{LIP}
\DeclareMathOperator{\Ch}{Ch}

\DeclareMathOperator{\Ker}{Ker}
\DeclareMathOperator{\Ent}{Ent}
\DeclareMathOperator{\Diam}{Diam}
\DeclareMathOperator{\Fix}{Fix}
\DeclareMathOperator{\Tang}{Tan}

\usepackage{enumitem}

\usepackage[usenames]{color}

\begin{document}
\title{Invariant measures and lower Ricci curvature bounds}
\author{Jaime Santos-Rodr\'iguez}
\address{Department of Mathematics, Universidad Aut\'onoma de Madrid and ICMAT CSIC-UAM-UCM-UC3M, Spain}
\email{jaime.santosr@estudiante.uam.es}
\thanks{The author was supported by research grants MTM2014-57769-3-P, and MTM2017-85934-C3-2-P (MINECO) and ICMAT Severo Ochoa Project SEV-2015-0554 (MINECO). }
\date{\today}


\subjclass[2010]{53C23, 53C21}
\keywords{Ricci curvature dimension, metric measure space,  isometry group}

\begin{abstract}
Given a metric measure space  $(X,d,\mathfrak{m})$ that satisfies the Riemannian Curvature Dimension condition,   $RCD^*(K,N),$  and a compact subgroup of isometries $G \leq \Iso(X)$ we prove that there exists a $G-$invariant measure, $\mathfrak{m}_G,$ equivalent to $\mathfrak{m}$ such that $(X,d,\mathfrak{m}_G)$ is still  a  $RCD^*(K,N)$ space.  We also obtain some applications to Lie group actions on $RCD^*(K,N)$ spaces. We look at homogeneous spaces, symmetric spaces and obtain dimensional gaps for closed subgroups of isometries.
\end{abstract}

\maketitle

\section{Introduction and statement of results}

Synthetic Ricci lower curvature bounds were introduced in the seminal papers of Lott-Villani \cite{LotVi} and Sturm \cite{SturmI}, \cite{SturmII}. This synthetic definition, known as Curvature-Dimension condition or $CD(K,N)$ is, broadly speaking, the convexity of certain functionals (called entropies) along geodesics on the space of probability measures.

In \cite{EKS} Erbar, Kuwada and Sturm defined another notion of Curvature-Dimension called Entropic Curvature Dimension condition, $CD^e(K,N).$ Under some technical assumptions  there is  equivalence of different notions of synthetic Ricci curvature bounds. For example we have that $CD^e (K,N)$ is equivalent to $RCD^*(K,N)$. (see Theorem 7 in \cite{EKS}). 

One thing to notice is that unlike the case of synthetic sectional curvature  a metric structure alone is not sufficient to talk about Ricci curvature. A reference measure is also required, that is we will work with metric measure spaces $(X,d,\mathfrak{m}).$  It is of interest then to see  how the measure $\mathfrak{m}$ and the metric $d$ interact with one another. 

In \cite{GuijSan} Guijarro and the author, and independently Sosa in \cite{Sosa}, studied the isometry group of $RCD^*(K,N)$ spaces. In both of these papers it is proved that the set of fixed points of an isometry must have zero $\mathfrak{m}-$measure. The main result of this article makes the connection between the metric and the measure deeper. More precisely we have:

\newtheorem*{thm:A}{Theorem \ref{teo:A}}
\begin{thm:A}
Let $(X,d,\mathfrak{m})$ be an $RCD^*(K,N)$ space and, $G\leq \Iso(X)$ a compact subgroup. Then there exists a $G-$invariant measure $\mathfrak{m}_G,$ equivalent to $\mathfrak{m},$ and such that $(X,d,\mathfrak{m}_G)$ is an $RCD^*(K,N)$ space.
 \end{thm:A}

 The strategy for finding the desired measure is the following: In \cite{Kell} Kell studied properties of optimal couplings with first marginal absolutely continuous with respect to $\mathfrak{m}.$ He obtained a measure rigidity result which states that any two essentially non-branching, qualitatively non-degenerate measures must be mutually absolutely continuous. Reference measures of $RCD^*(K,N)$ spaces in particular satisfy these conditions.  So then we first prove that the  reference measure $\mathfrak{m}$ is quasi-invariant with respect to $\Iso (X).$  That is, we prove that for every  isometry $g$ the measures $g_{\#}\mathfrak{m}$ and $\mathfrak{m}$ are mutually absolutely continuous. 
 
  Then we use the densities corresponding to the measures $g_{\#}\mathfrak{m}$ to construct an appropriate density function $\Psi_G$ such that the new measure $\mathfrak{m}_G := \Psi_G \mathfrak{m} $ is both invariant and  preserves the curvature bounds satisfied by the original measure $\mathfrak{m}.$

Studying Lie group actions has proved useful in Riemannian and Alexandrov geometry for finding about how the geometry and topology of a space affect one another. A couple of good sources to read on this are Grove's paper \cite{Grove}  or the more recent book by Dearricot et. al. \cite{Dear}.

 In \cite{GuijSan}, and indenpendently in \cite{Sosa}, it was shown that the isometry group of an $RCD^*(K,N)$ space is a Lie group. Therefore it makes sense now to study properties of Lie group actions on m.m.s.

In Lemma \ref{prop.preservstrat} we prove that the strafitication of a m.m.s. into $k-$regular sets ,  denoted as $\mathcal{R}_k,$ given by Mondino and Naber must be preserved by isometries (see Theorem \ref{teo.rect} for the stratification). We also look at the interaction of the orbits with the measure $\mathfrak{m}.$ Proposition \ref{prop.orbposmeas} states that the orbits can not have positive measure unless the action is transitive on the whole space. Furthermore, in Proposition \ref{prop.homogeneous} we prove that homogeneous spaces must in fact be Riemannian manifolds.   

In \cite{GalKellMonSosa} some extra assumptions were made in order to obtain information about the behaviour of  the $k-$regular sets  in the quotient space $X/G.$ (see Theorem $6.3$ \cite{GalKellMonSosa}).  We will also make one which we will refer to as the  \emph{isotropy condition } $\mathcal{I}.$ Briefly, this condition asks that for an appropriate $k-$regular point $p,$ the   orbits $G_p \cdot q$ of the isotropy group $G_p$  are rectifiable. (see Section \ref{Boundgroups} for a precise definition).  

In \cite{Kit}  Kitabeppu defined the Analytic dimension, $\underline{\dim} X$ , of an $RCD^*(K,N)$ space as the largest $k \in [1,N]\cap \mathbb{N}$ such that $\mathfrak{m}(\mathcal{R}_k)>0.$ For both Riemannian manifolds and Alexandrov spaces it is easy to see that the Analytic dimension coincides with the more classical notions used in the literature. However, in a more general setting  there does not seem to be any  relationship between the Hausdorff or the topological dimension and this new dimension.

Thanks to theorem \ref{teo:A} we can change the measure without losing the curvature bounds and extend some of the theorems of Galaz-Garc\'ia and Guijarro \cite{GalGui} for Alexandrov spaces. More specifically we prove:

\newtheorem*{thm:B}{Theorem \ref{teo:B}}
\begin{thm:B}
Let $(X,d,\mathfrak{m})$ be an $RCD^*(K,N)$ space,  such that $\underline{\dim } X =k \neq 4.$ Then there are no closed subgroups  $G \leq \Iso(X),$ satisfying the isotropy condition $\mathcal{I}$, and whose dimension lies in the interval:

$$\frac{1}{2}\kappa (\kappa-1)+1    <   \dim G <\frac{1}{2}\kappa(\kappa+1)  ,$$ where  $\kappa = \max\lbrace \dim_{T}\mathcal{R}_k, \underline{\dim}X\rbrace. $
\end{thm:B}

\newtheorem*{thm:C}{Theorem \ref{teo:C}}
\begin{thm:C}
Let $(X,d,\mathfrak{m})$ be an $RCD^*(K,N)$ space,  such that $\underline{\dim}X =k \neq 4.$ If  $\Iso(X)$ has dimension $\frac{1}{2}\kappa(\kappa-1)+1,$ where  $\kappa = \max\lbrace \dim_{T}\mathcal{R}_k, \underline{\dim}X\rbrace, $    and satisfies the isotropy condition $\mathcal{I}$ then $(X,d)$ is isometric to a Riemannian manifold.
\end{thm:C}

In the $4-$dimensional case, we prove separately  that if the isometry group has dimension $7$  or $8$ then the action must be transitive. Then by Ishihara's results \cite{Ishi} we obtain analogues to theorems \ref{teo:B} and \ref{teo:C}. We also prove the more general gap theorem of Mann \cite{Mann}. 

Finally we look at Symmetric spaces and prove, just as in the case of Alexandrov spaces, that $RCD^*(K,N)$ spaces that are symmetric must be Riemannian manifolds.

\subsection*{Acknowledgements}
The author would like to thank his advisor Prof. Luis Guijarro for helpful comments on earlier versions of this manuscript.

\section{Preliminaries}

This section will be devoted to introducing the notation and structural results that we will use throughout the paper.  A metric measure space $(X,d,\mathfrak{m})$ will consist of a  complete, separable and geodesic metric space $(X,d)$ and a Radon measure $\mathfrak{m}$ on the Borel $\sigma-$algebra $\mathcal{B}(X).$ Let $\mathbb{P}(X)$ denote the space of probability measures on $\mathcal{B}(X)$ and 
$$\mathbb{P}_2(X) := \left\lbrace \mu \in \mathbb{P}(X) \,|\,  \int_X d^2(x_0,x)d\mu < \infty \text{ for some (and hence all) } x_0 \in X   \right\rbrace $$
 the space of probability measures with finite second moments. 

\subsection{Optimal transport}

Given measures $\mu, \nu \in \mathbb{P}_2(X)$ a coupling between them  is a probability measure $\pi \in \mathbb{P}(X\times X)$ such that
$$p_{1\#}\pi = \mu,\quad  p_{2\#}= \nu, $$ where $p_i : X\times X \rightarrow X,$ $i= 1,2$
 are the projections onto the first and second factors respectively.  Using these we define the $L^2-$Wassertein distance, $\mathbb{W}_2,$ as:
 $$\mathbb{W}_2^2 (\mu,\nu) := \inf\left\lbrace \int_{X\times X} d^2(x,y) d\pi(x,y) \,|\, \pi \text{ is a coupling between } \mu \text{ and } \nu \right\rbrace. $$ Since $d^2$ is lower semicontinuous it follows that  the infimum can be replaced by a minimum.  It is known that  $(\mathbb{P}_2(X),\mathbb{W}_2)$ inherits good properties from $(X,d)$ such as being complete, separable and geodesic.

We will denote by $\Geo(X)$ the space of geodesics of $X$ and equip it with the topology of uniform convergence. For $t \in [0,1]$ we define the evaluation map $e_t : \Geo(X) \rightarrow X,$ $\gamma \mapsto \gamma_t.$ If we take a geodesic $\left(\mu_t \right)_{t \in [0,1]}$ in $\mathbb{P}_2(X),$ then there exists $\pi \in \mathbb{P}(\Geo(X))$ with $e_{t\#}\pi = \mu_t$ for all $t \in [0,1]$ and 
\begin{align*}
\mathbb{W}_2^2 (\mu_s,\mu_t) &= \int_{\Geo(X)}d^2(\gamma_s,\gamma_t)d\pi(\gamma)\\
&=(s-t)^2\int_{\Geo(X)}\text{length}^2(\gamma)d\pi(\gamma) \quad \text{ for all } s,t \in [0,1].
\end{align*}
(See  Theorem $2.10 $ \cite{AmbGig} for details). The collection of all such measures $\pi$ will be denoted as $\OptGeo(\mu_0,\mu_1).$

A set $\Gamma \subset \Geo(X)$ is non-branching if given $\gamma, \gamma' \in \Gamma$ such that  $\gamma_t = \gamma'_t$ for some $t \in [0,1]$ implies that $\gamma = \gamma'.$

\begin{defi}
We will say that a m.m.s.  $(X,d,\mathfrak{m})$ is \emph{essentially non-branching} if given  $\mu_0, \mu_1 \in \mathbb{P}_2 (X)$ such that $\mu_0 \ll \mathfrak{m},$  there exists a unique optimal transport 
$\pi \in \mathbb{P}(\Geo(X))$ concentrated on a non-branching set of geodesics $\Gamma$ and such that  $\mu_t \ll \mathfrak{m}$ for all $t \in [0,1).$   
\end{defi}

\subsection{Relative entropy}\label{sec.RelEntropy}

Let $(X,d,\mathfrak{m})$ be a m.m.s. For a measure $\mu \in \mathbb{P}_2(X)$ we define its relative entropy with respect to $\mathfrak{m}$ as:
$$\Ent_{\mathfrak{m}}(\mu) := \int \rho \log \rho \,d\mathfrak{m}, $$ whenever $\mu = \rho \mathfrak{m}$ and $(\rho \log \rho)_{+}$ is integrable, otherwise we  set it as $+\infty.$ The subset of measures with finite entropy will be denoted as $D(\Ent_\mathfrak{m}).$ In the case that the reference measure $\mathfrak{m}$ is a probability measure,  Jensen's inequality ensures that  the entropy is nonnegative. For measures of infinite mass but satisfying the growth condition
\begin{equation}\label{eq.growthcond}
\int \exp(-cd^2(x_0,x))d\mathfrak{m} < \infty,
\end{equation}
for some $c>0,$ and $x_0\in X,$ it is known (see for example Section $2$ of  \cite{AmbGigMonRaj} and references therein) that  $\Ent_{\mathfrak{m}}(\mu) > -\infty$ for all $\mu \in \mathbb{P}_2(X).$  

\subsection{Curvature-Dimension conditions}
In \cite{EKS} Erbar, Kuwada and Sturm studied convexity properties of the relative entropy of a m.m.s. $(X,d,\mathfrak{m}).$  They defined a new curvature dimension condition, called entropic curvature dimension condition, which allowed them to prove the equivalence between several previously defined notions of curvature bounds such as the ones given by Lott-Sturm-Villani and by Bakry-\'Emery. For details see Theorem $7$ in \cite{EKS}.

Given $K \in \mathbb{R},$ and $N \in [1, \infty)$   we define the distortion coefficients  $\sigma_{K,N}^{(t)}(\theta)$ for $t \in [0,1]:$  
\begin{equation*}
\sigma_{K,N}^{(t)}(\theta):=
\begin{cases}
\infty & \text{if }\, K\theta^2\geq N\pi^2, \\
\frac{\sin(t\theta \sqrt{K/N})}{\sin(\theta\sqrt{K/N})} & \text{if }\, 0<K\theta^2<N\pi^2,\\
t & \text{if }\, K\theta^2=0,\\
\frac{\sinh(t\theta\sqrt{-K/N})}{\sinh(\theta\sqrt{-K/N})} & \text{if } K\theta^2<0.
\end{cases}
\end{equation*}

Let $S: (X,d)\rightarrow [-\infty, \infty]$ be a functional on $X,$ and denote its proper domain by $D(S):= \lbrace x\in X |\, S(x)<\infty \rbrace.$  For $N \in (0,\infty)$ we  define the functional  $\mathcal{U}_N: (X,d) \rightarrow [0,\infty],$
\begin{equation}\label{eq.deffunctional}
\mathcal{U}_N(x) := \exp\left(-\frac{1}{N}S(x) \right).
\end{equation}

\begin{defi}
Let $S: (X,d)\rightarrow [-\infty, \infty]$ be a functional on $X,$ and  $K \in \mathbb{R},$ $N \in (0,\infty).$ We will say that the functional $S$ is $(K,N)-$convex if and only if for each $x_0, x_1 \in X$ there exists a geodesic $\gamma$ between them such that for all $t \in [0,1]: $
\begin{equation}\label{eq.KNconvex}
\mathcal{U}_N(\gamma_t )\geq \sigma_{K/N}^{(1-t)}(d(\gamma_0,\gamma_1))\,\mathcal{U}_N(\gamma_0)+\sigma_{K/N}^{(t)}(d(\gamma_0,\gamma_1))\,\mathcal{U}_N(\gamma_1). 
\end{equation}
\end{defi}
If this condition holds for any pair of points $x_0, x_1 \in D(S)$ and any geodesic $\gamma$ connecting $x_0 $ to $x_1$ then we will say that $S$ is strongly $(K,N)-$convex. Observe that if  $\gamma_0, \gamma_1 \in D(S) $ then any geodesic $\gamma $ that satisfies inequality (\ref{eq.KNconvex}) lies in $D(S).$

Examples of $(K,N)-$convex functionals can be found in examples $2.5$ and $2.6$ in \cite{EKS}.

\begin{defi}\label{def.entropiccurv}
Let $(X,d,\mathfrak{m})$ be a m.m.s., we will say that it satisfies the Entropic Curvature Dimension condition, $CD^e(K,N)$ for some $K \in \mathbb{R}$ and $N \in (0, \infty)$ if  and only if given any two measures $\mu_0, \mu_1 \in D(\Ent_{\mathfrak{m}})$ there exists a geodesic $\left(\mu_t \right)_{t \in [0,1]}$ in  $D(\Ent_{\mathfrak{m}})$    such that for all  $t \in [0,1]:$
\begin{equation}\label{eq.entropyKNconvex}
\mathcal{U}_N(\mu_t) \geq \sigma_{K/N}^{(1-t)}(\mathbb{W}_2(\mu_0,\mu_1))\,\mathcal{U}_N(\mu_0)+\sigma_{K/N}^{(t)}(\mathbb{W}_2(\mu_0,\mu_1))\,\mathcal{U}_N(\mu_1).
\end{equation}
\end{defi}
If this holds for any geodesic $\left(\mu_t \right)_{t \in [0,1]}$ in $D(\Ent_{\mathfrak{m}})$ we will say that the m.m.s is a strong $CD^e(K,N)$ space.  

So then,  broadly speaking $(X,d,\mathfrak{m})$ is a $CD^e(K,N)$ space if the relative entropy $\Ent_\mathfrak{m}$ is $(K,N)-$convex. 

\subsection{Sobolev spaces and Cheeger energy}
Despite the generality in which we will be working, there is some first order differential structure available for metric measure spaces. Here we only present the basic notions that we will use; for a more detailed exposition one can look at \cite{GigliII}.

Let  $\LIP(X)$ be the space of real valued Lipschitz functions on $X.$ Given a Lipschitz  function $f$ denote its  finiteness domain by $D(f) = \lbrace x \in X | f(x)\in \mathbb{R} \rbrace; $  the local Lipschitz constant is defined as:

\begin{equation}\label{def.Lipconst}
\Lip f(x) := \limsup_{y\rightarrow x} \frac{|f(x)-f(y) |}{d(x,y)},
\end{equation}
by convention we take $\Lip f(x) = \infty$ if $x \notin D(f),$ and $\Lip f(x)=0$ if $x$ is isolated.

\begin{defi}\label{def.Cheegerenergy}
Let  $(X,d,\mathfrak{m})$ be a m.m.s.  For   $f \in L^2(\mathfrak{m})$ we define the Cheeger energy as:

\begin{equation}
 \Ch_{\mathfrak{m}} (f) :=\inf \left\lbrace \liminf_{n \rightarrow \infty} \frac{1}{2}\int |\Lip f_n|^2 d \mathfrak{m} \,|\, f_n \in \LIP(X),  f_n \rightarrow f  \text{ in } L^2 (\mathfrak{m})   \right\rbrace.
\end{equation}
We denote its domain as $D(\Ch_{\mathfrak{m}}) := \lbrace f \in L^2(\mathfrak{m}) \,|\,  \Ch_{\mathfrak{m}} (f) < \infty \rbrace $ 
\end{defi}
It can be proved that for $f \in D(\Ch_{\mathfrak{m}})$ if one looks at the optimal approximation in the definition then there exists a function in $L^2(\mathfrak{m}),$ called the minimal weak upper gradient,  $|\nabla f|_w$ such that
\begin{equation*}
\Ch_{\mathfrak{m}}(f) = \frac{1}{2}\int |\nabla f|_w^{2} d\mathfrak{m}.
\end{equation*}

The minimal weak upper gradient is of local nature in the sense that for all  $f,g \in W^{1,2}(\mathfrak{m})$ 
\begin{equation}\label{eq.locality}
|\nabla f|_w = |\nabla g|_w  \quad \mathfrak{m}- \text{a.e. on } \lbrace f=g\rbrace.  
\end{equation}
We also have that  minimal weak upper gradients satisfy the following Leibniz rule:
\begin{equation}
|\nabla fg|_w \leq |f||\nabla g|_w+|g||\nabla f|_w  \quad  \text{ for all } f,g \in W^{1,2}(\mathfrak{m})\cap L^{\infty}(\mathfrak{m}).
\end{equation}

If $(X,d,\mathfrak{m})$   satisfies a weak Poincar\'e inequality and the reference measure $\mathfrak{m}$ is locally doubling, then by Theorem $6.1$  in \cite{Cheeg} we have  that for any locally Lipschitz function $f$
\begin{equation*}
\Lip f (x) = |\nabla f|_w (x)   \quad \mathfrak{m}-\text{a.e. }  x \in X.
\end{equation*}

Let  $W^{1,2}(\mathfrak{m}):= D(\Ch_{\mathfrak{m}}) $ and equip it with the norm  $$\|f\|_{W^{1,2}}:= \|f\|_{L^2(\mathfrak{m})} + \int |\nabla f|_w^2d\mathfrak{m}.$$ In the case that this space is actually a Hilbert space we will say that $(X,d,\mathfrak{m})$ is infinitesimally Hilbertian.

In \cite{GigliII} Gigli studied  the differential structure  of m.m.s. in particular, under the assumption that $W^{1,2}(\mathfrak{m})$ is a Hilbert space then it is possible to obtain a pointwise version of the parallelogram law  ( see Corollary $3.4$ in that paper).
\begin{equation}\label{eq.parallelpoint}
|\nabla(f+g) |_w^2+|\nabla(f-g)|_w^{2} = 2(|\nabla f|_w^2 +|\nabla g|_w^2)  \quad \mathfrak{m}-\text{a.e. }  \text{ for all } f,g \in W^{1,2}(\mathfrak{m}).
\end{equation}

\subsection{Gromov-Hausdorff convergence}
Now we recall the different notions of Gromov-Hausdorff convergence that will be used in Section 
\ref{sec.applications}.

\begin{defi}
Let $(X,d_X, p_X),$ $(Y,d_Y, p_Y)$ be complete and separable  metric spaces. For $\epsilon >0 $ we will call a function  $f_\epsilon : B_{p_X}(1/\epsilon, X)\rightarrow B_{p_Y}(1/\epsilon, Y) $ a pointed $\epsilon$-Gromov-Hausdorff approximation if:
\begin{enumerate}
\item $f_\epsilon(p_X)= p_Y,$

\item for all $u,v \in B_{p_X}(1/\epsilon, X)$, $|\,d_X(u,v)-d_Y(f_\epsilon(u),f_\epsilon(v))\,|< \epsilon,$ and

\item for all $y \in B_{p_Y}(1/\epsilon, Y)$, there exists a point $x \in B_{p_X}(1/\epsilon, X)$ such that $d_Y(f_\epsilon(x),y) < \epsilon$.
\end{enumerate}
 Note that we do not assume that $f_\epsilon$ is continuous.
\end{defi}

\begin{defi}
If  $(X_k,d_k, \mathfrak{m}_k, p_k)$, ${k \in \mathbb{N}}$, and $(X_\infty,d_\infty,\mathfrak{m}_\infty,p_\infty)$ are pointed metric measure spaces, we will say that $X_k$ converges to $X_\infty$ in the pointed measured Gromov-Hausdorff topology if there exist Borel measurable $\epsilon_k$--Gromov-Hausdorff approximations such that $\epsilon_k\rightarrow 0$ and $$(f_{\epsilon_k})_{\#}\mathfrak{m}_k\rightharpoonup \mathfrak{m}_\infty$$ weakly.

\end{defi}

%
%
Given$r>0$ and $G$ acting on $X$ by isometries define $$G(r) := \lbrace \gamma \in G \,|\, \gamma p \in B_p (r) \rbrace.$$

\begin{defi}
Let $(X,d_X, p, G)$ and  $ (Y,d_Y,q, H)$ be two metric spaces and $G, H$ groups acting isometrically on $X$ and $Y$ respectively. An \emph{equivariant pointed  $\epsilon$--Gromov-Hausdorff approximation} is a triple $(f,\varphi,\psi)$ of maps 
\[
f: B_p(1/\epsilon, X)\rightarrow B_q(1/\epsilon,Y),
\quad 
 \varphi : G (1/\epsilon)\rightarrow H (1/\epsilon),
\quad 
\psi: H(1/\epsilon)\rightarrow G(1/\epsilon)
\]  
such that
\begin{enumerate}
\item $f(p)=q;$
\item the $\epsilon$--neighbourhood of $f\left(B_p(1/\epsilon,X)\right)$ contains $B_q(1/\epsilon,Y);$
\item if $x,y \in B_p(1/\epsilon,X),$ then 
\[
\left|\,d_X(x,y)-d_Y(f(x),f(y))\,\right|< \epsilon;
\]
\item if $\gamma \in G (1/\epsilon)$ and both $x, \gamma x \in B_p(1/\epsilon,X), $ then $$d_Y(f(\gamma x), \varphi (\gamma)f(x)) < \epsilon; $$
\item if $\lambda \in H (1/\epsilon)$ and both $x, \psi (\lambda)x \in B_p(1/\epsilon,X),$ then  $$d_Y(f(\psi (\lambda)x), \lambda f(x)) < \epsilon. $$
\end{enumerate} As usual, we do not assume that $f$ is continuous or that $\varphi, \psi$ are group homomorphisms.
\end{defi}

\subsection{Structure theory}

Now we state the structural results known for the m.m.s. that will be of interest to us.
First, notice that by Theorem $7$  of \cite{EKS}, the Entropic Curvature Dimension condition \eqref{def.entropiccurv} coupled with $W^{1,2}(\mathfrak{m})$ being a Hilbert space  is equivalent to the  Riemannian Curvature Dimension condition, $RCD^*(K,N).$ So from now on we will always assume that $(X,d,\mathfrak{m}) $ satisfies both of these conditions unless otherwise stated.  

First, without assuming anything on $W^{1,2}(\mathfrak{m}),$ we have the following information on $(X,d,\mathfrak{m}):$
\begin{teo}{\bf (Sturm \cite{SturmII})}\label{teo.atoms}
Let $(X,d,\mathfrak{m})$ be a $CD^*(K,N)$ space.Then:
\begin{enumerate}
\item  $(X,d)$ has Hausdorff dimension bounded above by $ N.$

\item $(X,d)$ is proper, in particular, it is locally compact.

\item  The measure $\mathfrak{m}$ is locally doubling and has no atoms.

\item The measure $\mathfrak{m}$ satisfies the growth condition (\ref{eq.growthcond}), $$\int \exp (-c d^2(x_0,x))  d\mathfrak{m} < \infty,  \quad  \text{ for some }  x_0\in X, c>0.$$
\end{enumerate}

\end{teo}

Now if we assume that $W^{1,2}(\mathfrak{m})$ is a Hilbert space then we can rule out the presence of excessive branching and Finsler geometries. Moreover, it is also stable under measured Gromov-Hausdorff convergence when coupled with the curvature dimension condition $CD^*(K,N).$ The proof of this in the compact case was done by Ambrosio, Gigli and Savar\'e \cite{AmbGigSavI},  the general case was proved by Gigli, Mondino and Savar\'e \cite{GigMonSav}:

\begin{teo}\label{teo.stability}
The class of  metric mesaure spaces that are  both $CD^*(K,N)$ and infinitesimally Hilbertian is precompact under measured Gromov-Hausdorff convergence. 
\end{teo}

Gigli, Rajala and Sturm proved that $RCD^*(K,N)$ spaces are essentially non-branching, that is that optimal transports are concentrated in sets of non-branching geodesics. More precisely they obtained: 

\begin{teo}{\bf (Gigli, Rajala, Sturm \cite{GigRajStu})}\label{teo.essnonbranch}
Let $(X,d,\mathfrak{m})$ be an $RCD^*(K,N)$ space. Then for every $\mu,\nu \in \mathbb{P}_2(X)$ with $\mu \ll \mathfrak{m}$ there exists a unique optimal transport  $\pi \in \OptGeo(\mu,\nu).$ Furthermore, $\pi$ is induced by a map and  concentrated on a set of non-branching geodesics. That is, there exists a Borel set $\Gamma \subset C([0,1],X) $ such that $\pi(\Gamma)=1$ and for every $t \in [0,1)$ the map $e_t: \Gamma \rightarrow X$ is injective.
\end{teo}

In \cite{GigliI} Gigli obtained an extension of Cheeger-Gromoll's Splitting theorem for synthetic Ricci curvature bounds.

\begin{teo}{\bf (Splitting theorem)}\label{teo.split}
Let $(X,d, \mathfrak{m})$ be a $RCD^*(0,N)$ containing a line. Then it is isomorphic to the product of $(\mathbb{R},d_E,\mathcal{L}^1)$ and another space $(Y,d_Y,\nu).$ Moreover:
\begin{itemize}
\item If $N\geq 2,$ then $(Y,d_Y,\nu)$ is a $RCD^*(0,N-1),$

\item if $N\in [1,2),$ then $Y$ is a point.
\end{itemize}
\end{teo}

For a pointed metric measure space $(X,d,\mathfrak{m},p)$ and a constant $\lambda >0.$ We can rescale the metric by taking $\lambda d$ and normalize the measure in the following way
$$ \mathfrak{m}_\lambda := \left(  \int_{B_{p}(\lambda^{-1})}1-\lambda d(p,\cdot)d\mathfrak{m}\right)^{-1}\mathfrak{m}. $$

Given a sequence $\lambda_n \rightarrow \infty$ we can consider a sequence of rescaled and normalized pointed metric measure spaces $\lbrace (X_n, d_n,\mathfrak{m}_n,p_n ) \rbrace_{n \in \mathbb{N}}$

We define the tangent at a point $p \in X$ as:
$$\Tang(X,p):= \lbrace (Y,d_Y,\nu, y) |\, (X_n, d_n,\mathfrak{m}_n,p_n )\rightarrow   (Y,d_Y,\nu, y)   \rbrace $$

For  any  $\lambda >0,$ the space $(X, \lambda d, \mathfrak{m})$ is a $RCD^*(\lambda^{-2}K,N) $ space provided $(X,d,\mathfrak{m})$ is an $RCD^*(K,N)$ space. Then the stability theorem \ref{teo.stability}  implies that for any $p \in X$  the set $\Tang(X,p)$ is non-empty and consists of $RCD^*(0,N)$ spaces.
 
\begin{defi}
Let  $k\in [1,N]\cap \mathbb{N}.$ A point $x \in X$ is called $k-$regular if   $\Tang(X,x)= \lbrace (\mathbb{R}^k,d_E,\mathcal{L}^k,0)\rbrace.$ We denote the set of $k-$regular points as $\mathcal{R}_k.$  
\end{defi}


In \cite{MonNab}  Mondino and Naber   proved that $\mathfrak{m}-$a.e. point in $X$ has a unique Eulidean tangent, however the dimension of these tangents might vary from point to point.

\begin{teo}{\bf ( Mondino, Naber \cite{MonNab})}\label{teo.rect}
Let $(X,d,\mathfrak{m})$ be an $RCD^*(K,N)$ space for some $K \in \mathbb{R},$ $N \in (1, \infty).$ Then $\mathfrak{m}(X-\cup_{k=1}^{\lfloor N \rfloor}\mathcal{R}_k )=0. $ Furthermore, every $\mathcal{R}_k$ is $k-$rectifiable, i.e. it can be covered up to a $\mathfrak{m}-$negligible set by a countable collection of Borel subsets which are bi-Lipschitz equivalent to Borel subsets of $\mathbb{R}^k.$ 
\end{teo}

In the case of Ricci limits, Colding and Naber \cite{ColNab} showed  that there exists $l \in [1, N]\cap \mathbb{N}$ such that $\mathfrak{m}(X-\mathcal{R}_l)=0.$  It is not known for general $RCD^*(K,N)$ spaces if there exists a unique stratum of positive measure. However Kitabeppu and Lakzian \cite{KitLak} studied so called low dimensional $RCD^*(K,N)$ spaces and obtained the following characterization:

\begin{teo}\label{teo.lowdim}
Let $(X,d,\mathfrak{m})$ be an $RCD^*(K,N)$ space for some $K \in \mathbb{R}$ and $N \in (1,\infty).$ The following statements are equivalent:
\begin{enumerate}
\item $\mathcal{R}_1 \neq \emptyset,$

\item $\mathcal{R}_j = \emptyset$ for any $j \geq 2,$

\item $\mathfrak{m}(\mathcal{R}_j)=0$ for any $j\geq 2,$

\item $X$ is isometric to $\mathbb{R},$ to $[0,\infty),$ to $\mathbb{S}^1(r) $ for $r>0,$ or to $[0,a]$ for $a>0.$
\end{enumerate}
\end{teo}

Kitabeppu introduced in \cite{Kit} a new concept of dimension that coincides with previous definitions made by Colding and Naber for Ricci limit spaces \cite{ColNab} and by Han \cite{Han}.

\begin{defi}{(\bf Analytic dimension \cite{Kit}.)}
Let $(X,d,\mathfrak{m})$ be an $RCD^*(K,N)$ space. The analytic dimension of $X,$ $\underline{\dim} X,$ is defined as the largest number $k$ such that $\mathcal{R}_k\neq \emptyset$ and $\mathcal{R}_l =\emptyset$ for any $l>k.$ Equivalently, $\underline{\dim}X$ is the largest number $k$ such that $\mathfrak{m}(\mathcal{R}_k)>0.$ 
\end{defi}

From Theorem \ref{teo.atoms} we know  that the Hausdorff dimension, $\dim_{\mathcal{H}}X$ ,  of  an $RCD^*(K,N)$ space  is bounded above by $N.$   Since these spaces are complete and separable,  it follows (see discussion in \cite{GuijSan}) that for the  topological dimension we have  $\dim_T X \leq \dim_{\mathcal{H}}X \leq N.$


Using the concept of analytic dimension it was possible to obtain the lower semicontinuity of tangent cones in $RCD^*(K,N)$ spaces.

\begin{teo}{\bf (Lower semicontinuity of the analytic dimension \cite{Kit}.)}\label{thm.lscdim}
Let $(X_n,d_n,\mathfrak{m}_n,x_n)$ be a sequence of $RCD^*(K,N)$ spaces converging in the pointed measured Gromov-Hausdorff topology to  $(X_\infty,d_\infty,\mathfrak{m}_\infty,x_\infty).$ Then 
$$\underline{\dim}(X_\infty,d_\infty,\mathfrak{m}_\infty)\leq \liminf_{n\rightarrow \infty}\underline{\dim}(X_n,d_n,\mathfrak{m}_n).  $$
\end{teo}

Properties of the optimal transports with first marginal absolutely continuous with respect to the reference measure were studied by Kell in \cite{Kell}. He proved that changes in the reference measure, under the assumption that the measures are essentially non-branching and qualitatively non-degenerate, occur in the same measure class. In particular for $RCD^*(K,N)$ the following result was obtained:

\begin{teo}{\bf (Kell \cite{Kell})}\label{teo.equivmeas}
 If $(X,d,\mathfrak{m}_1)$ and $(X,d,\mathfrak{m}_2)$ are both $RCD^*(K,N)$ spaces for $N \in [1,\infty),$ then $\mathfrak{m}_1$ and $\mathfrak{m}_2$ must be mutually absolutely continuous. 
\end{teo} 
 
\subsection{Isometries, group actions, and quotients} 
n this section we will discuss some properties of the group of isometries of  a  $RCD^*(K,N)$ space.
A priori one would assume that the metric structure and the measure have no relationship at all; however we have the following result.

\begin{prop}{\bf (\cite{GuijSan}, \cite{Sosa})}\label{prop.fix}
If $g\in \Iso(X),$   the set $\Fix(g) := \lbrace x \in X |\,  gx=x \rbrace$ has $\mathfrak{m}-$measure zero.
\end{prop}

The following was obtained in \cite{GuijSan}, and independently by Sosa \cite{Sosa}.

\begin{teo}{\bf (Lie group structure)}
Let $(X,d,\mathfrak{m})$ be an $RCD^*(K,N)$ space.
Then the  isometry group, $\Iso(X),$ is a Lie group.
\end{teo}

Later on, Galaz-Garc\'ia, Kell, Mondino and Sosa studied Lie group actions on $RCD^*(K,N)$ spaces by measure preserving isometries. Amongst other results they obtained the following:

\begin{teo}{\bf (Quotients \cite{GalKellMonSosa}).}\label{teo.quotient}
Let $(X,d,\mathfrak{m})$ be an $RCD^*(K,N)$ space and $G\leq \Iso(X)$ a compact subgroup of measure preserving isometries. Then the quotient space $(X/G,d^*,\mathfrak{m}^*)$ is also an $RCD^*(K,N)$ space.  
\end{teo}



\begin{teo}{\bf (Principal orbit theorem \cite{GalKellMonSosa} ) }\label{teo.prinorbit}
Let $(X,d,\mathfrak{m})$ be an $RCD^*(K,N)$ space and $G \leq \Iso(X)$ a compact group of measure preserving isometries. Fix an orbit $G\cdot x \in X/G.$ Then for $\mathfrak{m}-$a.e. $y \in X$ there exists a unique $\bar{x}\in G\cdot x$ and a unique geodesic connecting $y$ to $\bar{x}$ such that $d(x,y)= d^*(G\cdot x, G\cdot y).$ In particular,  $G_y \leq G_{\bar{x}}$ and there exists a unique (up to conjugation) subgroup $G_{\min}\leq G$ such that for $\mathfrak{m}-$a.e. $y\in X$ the orbit $G\cdot y$ is homeomorphic to $G/G_{\min}.$ We will call such orbits principal. 
\end{teo}

\begin{rem}\label{rem.prinorbit}
Actually this theorem may be extended to closed subgroups $G$ of $\Iso(X)$  that don't necessarily preserve the measure.  By \cite{Pal} we know that the action is proper.  Then the isotropy groups are always compact, the quotient space $X/G$ is geodesic and we can  lift geodesics (see \cite{GuijSan} for details). This allows us to  follow the same argument as in \cite{GalKellMonSosa} and obtain the result. 
\end{rem}

First we recall the definitions and results used in the proof.
\begin{defi}{\bf ($c-$transform)}
Let $\psi: X \rightarrow \mathbb{R}\cup\lbrace \pm \infty \rbrace$ be a function. We define its $c-$transform, $\psi^{c}: X \rightarrow \mathbb{R}\cup\lbrace - \infty \rbrace$  , as 
$$\psi^c (x) := \inf_{y \in X} \lbrace d(x,y)^2-\psi(y) \rbrace.  $$
\end{defi}

We will say that a function $\varphi$ is $c-$concave if it is the $c-$transform of some other function $\psi,$ i.e. $\varphi = \psi^c$

\begin{defi}{\bf ($c-$superdifferential)}
Let $\varphi : X \rightarrow \mathbb{R}\cup \lbrace -\infty \rbrace$ be a $c-$concave function.  We define its $c-$superdifferential

$$\partial^c \varphi := \lbrace (x,y)\in X\times X \,|\, \varphi (x)+\varphi^c (y) = d(x,y)^2 \rbrace.  $$
For $x \in X$ define

$$\partial^c\varphi (x ) := \lbrace y \in X \,|\, (x,y)\in \partial^c \varphi \rbrace, $$ or equivalently  $y \in \partial^c \varphi (x) $ if and only if  
\begin{equation}\label{eq.equivsuperdiff}
\varphi (x) -d(x,y)^2 \geq \varphi (z) -d(x,z)^2, \quad  \forall z \in X.
\end{equation}
\end{defi}

\begin{lem}{\bf (\cite{GalKellMonSosa})}\label{lem.liftc}
The lift of a $c-$concave function is $c-$concave. That is, if $\varphi : X/G \rightarrow \mathbb{R}\cup \lbrace -\infty \rbrace$ is a $c-$concave function on $X/G$ then $\hat{\varphi}:X\rightarrow \mathbb{R}\cup\lbrace -\infty \rbrace $ defined by $\hat{\varphi}(x) := \varphi (G\cdot x)$ is $c-$concave in $X.$
\end{lem}

\begin{lem}{\bf (\cite{GalKellMonSosa})}\label{lem.superdiff}
Let $(X,d,\mathfrak{m})$ be an $RCD^*(K,N)$ space then for $\mathfrak{m}-$a.e. point $x\in X$ and every $c-$concave function $\varphi,$ the $c-$superdifferential contains at most one point. In particular for any $c-$concave function $\varphi$ and $\mathfrak{m}-$a.e. $x\in X$ there exists a unique geodesic connecting $x$ and $\partial^c\varphi (x),$ whenever  the set $\partial^c\varphi (x)$ is not empty. 
\end{lem}

Now we prove the Principal orbit theorem.
\begin{proof}[Proof of Remark 2.22]
Fix an orbit $G\cdot x,$  and define the function $\varphi: X \rightarrow \mathbb{R},$  $\varphi (y)= d^*  (G\cdot x, G\cdot y)^2 = d(G\cdot x, y)^2. $ By Lemma \ref{lem.liftc}  it is clear that this function is $c-$concave. 

By lifting a geodesic in $X/G$ we can find a point $\bar{x}\in G\cdot x$ such that it realizes $d(\bar{x},y)= d^*(G\cdot x, G \cdot y).$ Using the characterization of the $c-$superdifferential given by (\ref{eq.equivsuperdiff}) it is easy to check that $\bar{x} \in \partial^c \varphi (y).$ 
 
 Now, from Lemma \ref{lem.superdiff}   we have that for $\mathfrak{m}-$a.e. $y \in X$ the point $\bar{x}$ we previously obtained is unique and moreover that the geodesic connecting them is also unique.

Notice that $\partial^c \varphi (y)$ is $G_y-$invariant,  therefore  it follows that  from the $\mathfrak{m}-$a.e. uniqueness of the $c-$superdifferential that $G_y \leq G_{\bar{x}}.$ 

Finally take a point $x_0$ with minimal isotropy group $G_{x_0}.$ Then for $\mathfrak{m}-$a.e. point $y \in X,$  $G_y = gG_{x_0}g^{-1}.$ We conclude then that for $\mathfrak{m}-$a.e. $y \in X$ 
$$G\cdot x_0  \simeq G/G_{x_0}\simeq G/G_{y} \simeq G\cdot y,$$ are homeomorphic.
\end{proof}
 
\section{The measure $\mathfrak{m}$ is quasi-invariant.}

Recall that a measure $\mathfrak{m}$ is \emph{quasi-invariant} with respect to $\Iso(X)$ if for every $g\in \Iso(X) $ the measures $g_{\#}\mathfrak{m}$ and $\mathfrak{m}$ are mutually absolutely continuous. 

Let $(X,d,\mathfrak{m})$ be an $RCD^*(K,N)$ space for some $K \in \mathbb{R},$ $N \in [1, \infty),$ and take  $g \in \Iso(X).$  Our goal for this section will be to prove that the m.m.s. $(X,d,g_{\#}\mathfrak{m})$ is both infinitesimally Hilbertian and $CD^e(K,N).$ We begin with the observation that 
$\Iso(X)$ acts isometrically on $(\mathbb{P}_2(X), \mathbb{W}_2)$ by
\begin{align*}
\Iso(X)&\times \mathbb{P}_2(X) \rightarrow \mathbb{P}_2(X) \\
& (g,\mu) \longmapsto  g_{\#}\mu.
\end{align*}
Next, notice that  if $\mu \in \mathbb{P}_2(X,\mathfrak{m}),$ $\mu= \rho \mathfrak{m},$ then $g_{\#}\mu \in \mathbb{P}_2(X,g_{\#}\mathfrak{m})$ and for any $E\in \mathcal{B}(X): $
$$g_{\#}\mu (E) = \mu (g^{-1}E) = \int_{g^{-1}E} \rho \,d\mathfrak{m} =  \int_{g^{-1}E} \rho \,d(g^{-1}\circ g)_{\#}\mathfrak{m} = \int_{E}\rho \circ g^{-1}\,dg_{\#}\mathfrak{m}. $$
hence $g_{\#}\mu= \rho\circ g^{-1} g_{\#}\mathfrak{m}.$

From this we have that the relative entropies behave as:
\begin{equation}\label{eq.entropyg}
\Ent_{g_{\#}\mathfrak{m}}(g_{\#}\mu)= \int \rho \circ g^{-1}\log (\rho\circ g^{-1}) \,dg_{\#}\mathfrak{m} = \int \rho \log \rho \,d\mathfrak{m}= \Ent_{\mathfrak{m}}(\mu), 
\end{equation}
and  $g_{\#} D(\Ent_{\mathfrak{m}})= D(Ent_{g_{\#} \mathfrak{m}} ).$
\begin{prop}
$(X,d,g_{\#}\mathfrak{m})$ is $CD^e(K,N).$
\end{prop}
\begin{proof}
Let $\mu_0, \mu_1 \in \mathbb{P}_2 (X,g_{\#}\mathfrak{m}).$ Then $g_{\#}^{-1}\mu_0, g_{\#}^{-1}\mu_1 \in \mathbb{P}_2 (X,\mathfrak{m}), $ notice that since $(X,d,\mathfrak{m})$ is an $RCD^*(K,N)$ space by Theorem \eqref{teo.essnonbranch} there is a unique geodesic connecting $g_{\#}^{-1}\mu_0$ and $g_{\#}^{-1}\mu_1.$  Furthermore, we write the unique geodesic between them as  $(g_{\#}^{-1}\mu_t )_{t \in [0,1]} \subset \mathbb{P}_2(X,\mathfrak{m}).$  From (\ref{eq.entropyg}) we obtain that $\Ent_{g_{\#}\mathfrak{m}}(\mu_t)= \Ent_{\mathfrak{m}}(g_{\#}^{-1}\mu_t)$ and therefore
\begin{align*}
\mathcal{U}_N^{g_{\#}\mathfrak{m}} (\mu_t)&= \exp\left(-\frac{1}{N}\Ent_{g_{\#}\mathfrak{m}}(\mu_t) \right) \\
&= \exp\left(-\frac{1}{N}\Ent_{\mathfrak{m}}(g_{\#}^{-1}\mu_t) \right) = \mathcal{U}_N^{\mathfrak{m}}(g_{\#}^{-1}\mu_t).
\end{align*}
Since $g$ is an isometry $\mathbb{W}_2(\mu_0,\mu_1)=\mathbb{W}_2(g_{\#}^{-1}\mu_0,g_{\#}^{-1}\mu_1).$ So we conclude that

$$\mathcal{U}_N^{g_{\#}\mathfrak{m}} (\mu_t) \geq \sigma_{K/N}^{(1-t)}(\mathbb{W}_2(\mu_0,\mu_1))\mathcal{U}_N^{g_{\#}\mathfrak{m}} (\mu_0)+\sigma_{K/N}^{(t)}(\mathbb{W}_2(\mu_0,\mu_1))\mathcal{U}_N^{g_{\#}\mathfrak{m}} (\mu_1),  $$as desired.
\end{proof}

\begin{prop}\label{prop.isominfHilb}
$(X,d,g_{\#}\mathfrak{m})$ is infinitesimally Hilbertian.
\end{prop}
\begin{proof}
We define the map: 
\begin{align*}
W^{1,2}(&\mathfrak{m})\rightarrow W^{1,2}(g_{\#}\mathfrak{m}) \\
&f \longmapsto f\circ g^{-1}.
\end{align*}

It is clear that $\|f\|_{L^2(\mathfrak{m})}= \|f\circ g^{-1}\|_{L^2(g_{\#}\mathfrak{m})}.$ As for the Cheeger energy, let $f \in W^{1,2}(\mathfrak{m})$ and consider a sequence $(f_n)_{n\in \mathbb{N}}\subset \LIP (X)$ such that $f_n \rightarrow f$ in $L^2(\mathfrak{m}).$
Notice that
$$\int |f_n\circ g^{-1}-f\circ g^{-1}|^2dg_{\#}\mathfrak{m} = \int|f_n-f|^2d\mathfrak{m}, $$  therefore $f_n\circ g^{-1}\rightarrow f\circ g^{-1} $ in $L^2(g_{\#}\mathfrak{m}).$

Also it is easy to see that $\Lip f_n\circ g^{-1} (x) = \Lip f_n (g^{-1}x)$ so
$$\int |\Lip f_n\circ g^{-1}|^2dg_{\#}\mathfrak{m}= \int |\Lip f_n|^2d\mathfrak{m}. $$
If on the other hand we start with a sequence of functions $(h_n)_{n\in\mathbb{N}}\subset \LIP (X)$ such that 
$h_n \rightarrow f\circ g^{-1}$ in $L^2(g_{\#}\mathfrak{m})$ we have analogous observations. Namely that $h_n \circ g \rightarrow  f$ in $L^2 (\mathfrak{m})$ and $$\int |\Lip h_n \circ g|^2 d\mathfrak{m} = \int |\Lip h_n|^2 d g_{\#}\mathfrak{m}.$$ 
We  conclude that $$\Ch_{\mathfrak{m}}(f) = \Ch_{g_{\#}\mathfrak{m}}(f\circ g^{-1}). $$
So then the map $f \mapsto f\circ g^{-1}$ is an isometry and we have the result.
 \end{proof}

By Theorem $7$ in \cite{EKS} a space  that is both $CD^e(K,N)$ and infinitesimally Hilbertian is equivalent to being $RCD^*(K,N).$ It follows from Theorem \ref{teo.atoms} that the reference measure $g_{\#}\mathfrak{m}$ satisfies the growth condition (\ref{eq.growthcond}) (see Section \eqref{sec.RelEntropy} ), for some non-zero constant and some point in $X.$ Specifically the same constant as for $\mathfrak{m}$  and the point $gx_0$ will work out. We will need to prove that for any compact subgroup $G \leq \Iso(X)$ there exists some constant $c' >0$ such that 
$$\int_{X} \exp(-cd^2(x_0,x))\,dg_{\#}\mathfrak{m} < \infty, \text{ for all } g \in G.$$ 
Let $\delta := \Diam G\cdot x_0.$ Notice that for all $x \in X - B_{x_0}(\delta, X),$  
$$d(gx_0,x)\leq d(gx_0, x_0)+d(x_0,x)\leq 2d(x_0,x)$$ so then
\begin{align*}
\int_{X - B_{x_0}(\delta, X)}\exp(-4cd^2(x_0,x))dg_{\#}m &\leq \int_{X - B_{x_0}(\delta, X)}\exp(-cd^2(gx_0,x))dg_{\#}m\\
&\leq \int_{X} \exp(-cd^2(gx_0,x))dg_{\#}m\\
&=\int_{X} \exp(-cd^2(x_0,x))dm < \infty.
\end{align*}
As for the integral over the ball
\begin{align*}
\int_{B_{x_0}(\delta, X)}\exp(-4cd^2(x_0,x))dg_{\#}\mathfrak{m}\leq M g_{\#}\mathfrak{m}(B_{x_0}(\delta, X)) <\infty,
\end{align*}
where $M:= \max\lbrace \exp(-4cd^2(x_0,x)) \,|\, x\in B_{x_0}(\delta, X) \rbrace.$
 So $g_{\#}\mathfrak{m}$ satisfies the growth condition (\ref{eq.growthcond}) for $x_0$ and $4c.$ 
\section{$G-$invariant measures.}

Let $(X,d,\mathfrak{m})$ be an $RCD^*(K,N)$ space, throughout this section $G$ will denote a compact subgroup of $\Iso(X).$  Recall that two measures are said to be equivalent if they are mutually absolutely continuous.

\begin{prop}
If $g \in \Iso (X)$ then $\mathfrak{m}$ and $g_{\#}\mathfrak{m}$ are equivalent.
\end{prop}
\begin{proof}
Let $g\in \Iso(X).$  In the previous section we proved that $(X,d, g_{\#}\mathfrak{m})$ is an $RCD^*(K,N)$ space. Since $(X,d,\mathfrak{m})$ is also an $RCD^*(K,N)$ space by Theorem \ref{teo.equivmeas}  $g_{\#}\mathfrak{m}$ and $\mathfrak{m}$ must be equivalent.
\end{proof}

 Denote the Radon-Nikodym derivative of $g_{\#}\mathfrak{m}$  by   $\varphi_g = \frac{d g_{\#}\mathfrak{m}}{d\mathfrak{m}}.  $ 
\begin{prop}\label{prop.cocycle}
For all $g, h \in \Iso(X),$   we have that the following cocycle condition holds:   $\varphi_{hg}= \varphi_g \circ h^{-1} \varphi_h.$  
\end{prop}

\begin{proof}
For any $A \in \mathcal{B}(X),$ 
\begin{align*}
\int_A \varphi_{hg} \,d\mathfrak{m}&= (hg)_{\#}\mathfrak{m}(A)= g_{\#}\mathfrak{m}(h^{-1}A)=\\ 
&\int_{h^{-1}A}\varphi_g \,d\mathfrak{m}= \int_{h^{-1}A}\varphi_g \,d(h^{-1}\circ h)_{\#}\mathfrak{m} = \int_A (\varphi_g \circ h^{-1})\,\varphi_h \,d\mathfrak{m}.
\end{align*}
\end{proof}

The main objective of this section will be to prove the following result:
\begin{thmp}\label{teo:A}
Let $(X,d,\mathfrak{m})$ be an $RCD^*(K,N)$ space and, $G\leq \Iso(X)$ a compact subgroup. Then there exists a $G-$invariant measure $\mathfrak{m}_G,$ equivalent to $\mathfrak{m},$  such that $(X,d,\mathfrak{m}_G)$ is an $RCD^*(K,N)$ space. 
\end{thmp}

A natural way to construct a $G-$invariant measure $\tilde{\mathfrak{m}}$ is just to take the mean of the measures $g_{\#}\mathfrak{m}$  with respect to $\mathbb{H},$ the Haar measure of $G.$ That is,
$$\tilde{\mathfrak{m}}:=  \int_{G}g_{\#}\mathfrak{m}\,d\mathbb{H}(g). $$It is easy to see that $\tilde{\mathfrak{m}}$ is a Radon measure  equivalent to the original  $\mathfrak{m}.$ Therefore   $\tilde{\mathfrak{m}}=\int_{G}g_{\#}\mathfrak{m}\,d\mathbb{H} = \Phi_G \mathfrak{m},$   

By exercise $6.10.72$ of \cite{Bog} we can find a  $\mathfrak{m}\otimes \mathbb{H}-$measurable function $F$  on $X\times G$ such that  for all $g\in G,$   $x\in X,$ $(x,g) \mapsto F(x,g)=  \varphi_g (x):= \frac{d g_{\#}\mathfrak{m} }{d\mathfrak{m}}(x). $ It is clear that $F$ is non-negative.
If $A \in \mathcal{B}(X),$  by Fubini-Tonelli 
\begin{align*}
\tilde{\mathfrak{m}}(A)= \int_A \Phi_G \,d\mathfrak{m} &=\int_G g_{\#}\mathfrak{m}(A)\,d\mathbb{H} = \\
& \int_G \left(\int_A \varphi_g \,d\mathfrak{m}\right)d\mathbb{H}= 
\int_A \left(\int_G F d\mathbb{H} \right)d\mathfrak{m}.
\end{align*}
It follows then that  $\Phi_G = \int_G \varphi_g \,d\mathbb{H}. $

Now we define the locally integrable function 

\begin{equation}\label{eq.psiphi}
\Psi_G (x) := \exp \left(\int_G\log \varphi_g (x) \,d\mathbb{H} \right) \leq \Phi_G (x).
\end{equation}
where the inequality follows from applying Jensen's inequality. The aim of this section will be to prove that the m.m.s. $(X,d, \mathfrak{m}_G),$ where  $\mathfrak{m}_G := \Psi_G \mathfrak{m},$ is an $RCD^*(K,N) $ space.  Notice that $\mathfrak{m}_G $ and $\mathfrak{m}$ are mutually absolutely continuous.

\begin{lem}
$\mathfrak{m}_G$ is  a $G-$invariant measure.
\end{lem}
\begin{proof}
The result follows from the cocycle condition proved in Proposition \ref{prop.cocycle}. If $h\in G,$ notice that $h_{\#}\mathfrak{m}_G=  (\Psi_G \circ h^{-1})\,\varphi_h \mathfrak{m}.$ Then
\begin{align*}
\left(\Psi_G \circ h^{-1}\right)\varphi_h &= \exp\left(\int_G \log \varphi_g \circ h^{-1} d\mathbb{H}(g) \right)\varphi_h \\
											 &= \exp\left(\int_G \log \varphi_{hg} \varphi_h^{-1} d\mathbb{H}(g) \right)\varphi_h \\
											 &= \exp\left(\int_G \log \varphi_{hg} d\mathbb{H}(g) \right)\varphi_h^{-1} \varphi_h = \Psi_G.
\end{align*}
\end{proof}

Now, before checking Curvature-Dimension condition,  let us look at the entropy functional. The next Lemma characterizes the domain of $\Ent_{\mathfrak{m}_G}.$

\begin{lem}
The entropy functional, $\Ent_{\mathfrak{m}_G}$, satisfies the following:
\begin{enumerate}
\item For all $\mu \in \mathbb{P}_2(X),$  $\Ent_{\mathfrak{m}_G}(\mu) \in (-\infty, \infty]$

\item If $g \in G$ and $\mu \in D(\Ent_{\mathfrak{m}_G}),$ then $\Ent_{\mathfrak{m}_G}(\mu)= \Ent_{\mathfrak{m}_G}(g_{\#}\mu) $ and  $\int \log\varphi_g \,d\mu =0.$

\end{enumerate}

\end{lem}

\begin{proof}[Proof of $(1)$]
  From previous discussions we have that there exist $c>0$ and $x_0 \in X$ such that
$$\int_X \exp(-cd^2(x_0,x))\,dg_{\#}\mathfrak{m} < \infty \quad\text{ for all } g \in G. $$ 
Integrating over $G$ and then recalling that $\Psi_G \leq \Phi_G$ we obtain
$$\int_X \exp (-cd^2(x_0,x))\,d\mathfrak{m}_G \leq \int_G\int_X \exp (-c d^2(x_0,x))\varphi_g \,d\mathfrak{m}\,d\mathbb{H}< \infty. $$ Hence $\mathfrak{m}_G$ satisfies the growth condition (\ref{eq.growthcond}) and by the argument made in Section $2$ of \cite{AmbGigMonRaj} we have that
$$\Ent_{\mathfrak{m}_G}(\mu) > -\infty \quad \quad \text{ for all } \mu \in \mathbb{P}_2 (X). $$
\end{proof}

\begin{proof}[Proof of $(2)$]

For $(2).$ Let $\mu \in D(\Ent_{\mathfrak{m}_G}).$ Since $\mathfrak{m}_G = \Psi_G\mathfrak{m}$ we have the following formula (see Section $7.1$ in \cite{AmbGigSavII})
\begin{equation}\label{eq.changevarmG}
\Ent_{\mathfrak{m}_G}(\mu) = \Ent_{\mathfrak{m}}(\mu) -\int \int_G \log \varphi_g \,d\mathbb{H}d\mu.
\end{equation}
Since $\mathfrak{m}$ satisfies the growth condition (\ref{eq.growthcond}) it follows that $\Ent_{\mathfrak{m}}(\mu) > -\infty,$ which means that $-\int_X \int_G \log \varphi_g \,d\mathbb{H} \,d\mu>-\infty $  and hence $\Ent_{\mathfrak{m}}(\mu) < \infty $ as well.  From this it follows that 
$-\int_X \int_G \log \varphi_g \,d\mathbb{H} \,d\mu< \infty.$ Given $g \in G$ we have the following formula for $\Ent_\mathfrak{m}:$
\begin{equation}\label{eq.changevarm}
\Ent_{\mathfrak{m}}(g_{\#}\mu) = \Ent_{\mathfrak{m}}(\mu) -\int \log \varphi_g \,d\mu
\end{equation}

By the $G-$invariance of $\mathfrak{m}_G,$ $\Ent_{\mathfrak{m}_G}(\mu) =  \Ent_{g_{\#}\mathfrak{m}_G}(\mu)=\Ent_{\mathfrak{m}_G}(g^{-1}_{\#}\mu).$ 
Finally, if $\mu \in D(\Ent_{\mathfrak{m}_G}),$ $g\in G, $ then by \eqref{eq.changevarmG}, \eqref{eq.changevarm}, and the $G-$invariance we have:
\begin{align*}
\Ent_{\mathfrak{m}_G}(\mu) &= \Ent_{\mathfrak{m}_G}(g_{\#}\mu) \\
											 &= \Ent_{\mathfrak{m}}(g_{\#}\mu)  -\int \int_G \log \varphi_h \,d\mathbb{H}\,d\mu \\
											 &=\Ent_{\mathfrak{m}}(\mu)-\int \log \varphi_g \,d\mu  -\int \int_G \log \varphi_h \,d\mathbb{H}\,d\mu \\
											 &= \Ent_{\mathfrak{m}_G}(\mu)-\int \log \varphi_g \,d\mu.  
\end{align*}
We conclude then that $\int \log \varphi_g \,d\mu=0.$
\end{proof}

\begin{teo}
The m.m.s. $(X,d, \mathfrak{m}_G)$ is a $CD^e (K,N)$ space.
\end{teo}
\begin{proof}
Let $\mu_0, \mu_1 \in D(\Ent_{\mathfrak{m}_G}).$ Since $\mu_0,\mu_1$ are absolutely continuous with respect to $\mathfrak{m}$ and $(X,d,\mathfrak{m})$ is an $RCD^*(K,N)$ space, by Theorem \ref{teo.essnonbranch}  there exists a unique geodesic $(\mu_t)_{t \in [0,1]}\subset \mathbb{P}_2(X) $ connecting them.

From the previous Lemma we have that for all $g \in G$

$$\int \log \varphi_g \,d\mu_0=\int \log \varphi_g \,d\mu_1=0. $$ 
By \eqref{eq.changevarm} it follows that $\Ent_\mathfrak{m}(\mu_i)=\Ent_{g_{\#}\mathfrak{m}}(\mu_i),$ $i=0,1.$ That is, the measures $\mu_0,\mu_1 \in D(\Ent_{g_{\#}\mathfrak{m}})$ for all $g \in G.$

Let $t \in [0,1];$ using inequality (\ref{eq.entropyKNconvex}) in the definition of $CD^e(K,N)$ and then integrating over $G,$ we get:
\begin{align*}
\int_G \Ent_{g_{\#}\mathfrak{m}}(\mu_t) \,d\mathbb{H} \leq &-N \log(\sigma_{K/N}^{(1-t)}(\mathbb{W}_2(\mu_0,\mu_1))\,\mathcal{U}_N(\mu_0) \\
&+\sigma_{K/N}^{(t)}(\mathbb{W}_2(\mu_0,\mu_1))\,\mathcal{U}_N(\mu_1)) < +\infty
\end{align*}
Now $\int_G \Ent_{g_{\#}\mathfrak{m}}(\mu_t) d\mathbb{H} = \Ent_{\mathfrak{m}}(\mu_t)-\int_G\int\log\varphi_gd\mu_td\mathbb{H} = \Ent_{\mathfrak{m}_G}(\mu_t).$ Therefore $\mu_t \in D(\Ent_{\mathfrak{m}_G})$ and by the previous Lemma $\int \log\varphi_g d\mu_t =0$ for all $g \in G.$ 

Finally, multiply both sides of the inequality by $-1/N$ and apply the exponential. It is now clear that the functional $\Ent_{\mathfrak{m}_G}$ is $(K,N)-$convex.
\end{proof}

\begin{rem}\label{rem.doublingPoincare}
It is also easy to convince oneself that $(X,d,\mathfrak{m}_G)$ is also essentially non-branching.  Let $\mu_0,\mu_1 \in \mathbb{P}_2 (X)$ with $\mu_0 \ll \mathfrak{m}_G.$ Since $\mathfrak{m}_G$ and $\mathfrak{m}$ are equivalent then $\mu_0 \ll \mathfrak{m}.$ Then by Theorem \ref{teo.essnonbranch} we have that there exists a unique $\pi \in \OptGeo (\mu_0,\mu_1)$ that is concentrated on a set of non-branching geodesics.

 So by Theorem  $3.12$ in \cite{EKS} this m.m.s is $CD^*(K,N).$ Then we have that $\mathfrak{m}_G$ is locally doubling and that $(X,d,\mathfrak{m}_G)$ supports a weak Poincar\'e inequality. (see \cite{BachStu})
\end{rem}

Before checking that the Cheeger energy  $\Ch_{\mathfrak{m}_G}$ is quadratic we need to make some observations on how we can approximate  functions in $W^{1,2}(\mathfrak{m}_G)$ conveniently.

Consider a m.m.s $(X,d,\mathfrak{m}),$ with $\mathfrak{m}$ locally doubling  and such that the minimal weak upper gradients of Lipschitz functions  coincide $\mathfrak{m}-$a.e. with their local Lipschitz constant.  Let $\Lip_b(X),$ $\Lip_c(X)$ denote the space of bounded Lipschitz functions and compactly supported Lipschitz functions respectively. 
 Let $f \in \Lip_b(X)\cap L^2(\mathfrak{m})$ and consider a sequence of compact sets $(K_n)_{n\in \mathbb{N}} \subset X $  such that $\chi_{K_n} \uparrow 1.$ Notice that each $K_n$ has finite measure. 

Let $\phi : \mathbb{R}\rightarrow [0,1]$ be a $1-$Lipschitz, non-increasing function such that 
$\phi \equiv 1$ on $[0,1/3]$ and $\phi \equiv 0$ on $[2/3,\infty).$ With this we define $\phi_n := \phi (d(x,K_n)),$ this function is upper semicontinuous, Lipschitz, bounded and 
\begin{itemize}
\item $\phi_n(x)= |\nabla \phi_n|_w(x)=0$   on $\lbrace x \in X | d(x,K_n)> 2/3 \rbrace,$

\item $|\nabla \phi_n|_w = \Lip \phi_n$ on $K_n.$
\end{itemize}
It is clear that $\phi_nf \in \Lip_c(X)\cap L^2(\mathfrak{m}).$ For the sequence $(\phi_n f)_{n \in \mathbb{N}}$ we have that:
\begin{align*}
 \int |\phi_n f-f|^2 \,d\mathfrak{m} = &\int_{\lbrace x \in X \,|\, d(x,K_n)\leq 1/3 \rbrace } |\phi_n f-f|^2 \,d\mathfrak{m} \\ &+ \int_{\lbrace x \in X \,|\, d(x,K_n) > 1/3 \rbrace } |\phi_n f-f|^2 \,d\mathfrak{m}.
\end{align*}
The first integral on the RHS is equal to  zero because for all $n \in \mathbb{N}$ $\phi_n \equiv 1$ in $\lbrace x \in X \,|\, d(x,K_n)\leq 1/3 \rbrace.$ The second integral may be bounded from above by $\|f\|_{\infty}\mathfrak{m}(\lbrace x \in X \,|\, d(x,K_n) > 1/3 \rbrace).$ 

Since $\chi_{K_n}\uparrow 1$ it follows that $\mathfrak{m}(\lbrace x \in X \,|\, d(x,K_n) > 1/3 \rbrace) \rightarrow 0$ as $n \rightarrow \infty.$ Hence $\phi_n f \rightarrow f$ strongly in $L^2(\mathfrak{m}).$ 

For the sequence $(|\nabla \phi_n f|_w)_{n \in \mathbb{N}}$ we have a similar argument. 
\begin{align*}
\int_X ||\nabla \phi_n f|_w-\Lip f |^2 \,d\mathfrak{m} = &\int_{\lbrace x \in X \,|\, d(x,K_n)\leq 1/3 \rbrace} ||\nabla \phi_n f|_w-\Lip f |^2 \,d\mathfrak{m}\\ &+ \int_{\lbrace x \in X \,|\, d(x,K_n) > 1/3 \rbrace}||\nabla \phi_n f|_w-\Lip f |^2 \,d\mathfrak{m}.
\end{align*}
The first integral on the RHS is zero because $\phi_n  \equiv 1$ on $\lbrace x \in X \,|\, d(x,K_n)\leq 1/3 \rbrace$ and  $|\nabla f|_w = \Lip f$ $\mathfrak{m}-$a.e. The second integral can be bounded from above by $\|\Lip f\|_{L^2(\mathfrak{m})}^2 \mathfrak{m}(\lbrace x \in X \,|\, d(x,K_n) > 1/3 \rbrace).$  This is clear because $f \in W^{1,2}(\mathfrak{m}).$ As before $\mathfrak{m}(\lbrace x \in X \,|\, d(x,K_n) > 1/3 \rbrace) \rightarrow 0$ as $n \rightarrow \infty.$ Therefore it follows that  $|\nabla \phi_n f|_w \rightarrow \Lip f$ strongly in $L^2 (\mathfrak{m}).$  
\begin{teo}
The m.m.s. $(X,d, \mathfrak{m}_G)$ is infinitesimally Hilbertian.
\end{teo}
\begin{proof}
Let $f,g \in \Lip_c (X)\cap L^2(\mathfrak{m}_G);$ notice that $f,g \in \Lip_c (X)\cap L^2(\mathfrak{m}).$ Then we have the pointwise parallelogram law (\ref{eq.parallelpoint}) $\mathfrak{m}-$a.e.  By Theorem $6.1$ in \cite{Cheeg}, for  $\mathfrak{m}$ and $\mathfrak{m}_G$ the minimal weak upper gradients of $f$ and $g$ are  just the corresponding Lipschitz constants.  Since $\mathfrak{m} \ll \mathfrak{m}_G \ll \mathfrak{m}$ it follows that we also have the pointwise parallelogram law  (see also equation \eqref{eq.parallelpoint})
$$\nabla(f+g) |_w^2+|\nabla(f-g)|_w^{2} = 2(|\nabla f|_w^2 +|\nabla g|_w^2)  \quad \mathfrak{m}_G-\text{a.e. }$$

 Therefore by integrating we obtain that $\Ch_{\mathfrak{m}_G}$ is quadratic on  $ \Lip_c (X)\cap L^2(\mathfrak{m}_G).$ By the approximation done before we have this on $ \Lip_b (X)\cap L^2(\mathfrak{m}_G),$ and by Proposition $4.1$ in \cite{AmbGigSavII} this last set is dense in $W^{1,2}(\mathfrak{m}_G).$  
\end{proof}
With this we have finished the proof of Theorem $A.$ 
\begin{rem}
The measure $\mathfrak{m}_G$ is not unique in the sense that there might be different $G-$invariant measures that satisfy the Riemannian Curvature Dimension condition. 
Consider $([-1,1], d_E, \mathcal{L}^1),$  where $d_E$ is the Euclidean distance and $\mathcal{L}^1$ is the Lebesgue measure restricted to $[-1,1].$
From example $2.6$ $ ii)$ in \cite{EKS} we have that given $a \in [-1,1],$  $N >0,$ and $K >0$   the function
\begin{equation}
S_a (x) := -N \log (\cosh (d_E (x,a)\sqrt{-K/N})) 
\end{equation}
is $(K,N)-$convex. It clear that $S_a$ is also continuous and bounded from below. Therefore from Proposition $3.31$ in \cite{EKS}  we have that 
$$([-1,1],d_E, \exp (-S_a)\mathcal{L}^1) $$ is an $RCD^*(K,N+1)$ space.
The only non-trivial isometry of [-1,1] is the map $x \mapsto -x.$
We apply Theorem $A$ to $([-1,1],d_E, \exp (-S_a)\mathcal{L}^1)$ and obtain $G-$invariant measures  $\mathfrak{m}_{G,a},$ all of which are equivalent to one another but not equal, since given 
$a, a' \in [-1,1]$ such that $a \neq \pm a'$ we have that
$$\mathfrak{m}_{G,a} ([-1,1]) \neq \mathfrak{m}_{G,a'}([-1,1]). $$
\end{rem}

\section{Applications to Lie group actions}\label{sec.applications}
By Theorem \ref{teo:A} proved in the previous section we can safely  change the reference measure without losing the lower Ricci curvature bounds. So from here on we can assume without loss of generality that whenever we consider a compact subgroup $G\leq \Iso(X)$ it acts by measure preserving isometries on $(X,d,\mathfrak{m}).$

\subsection{Improved dimension bounds on the isometry group.}\label{Boundgroups}
Here we will provide refinements of the bounds obtained in \cite{GuijSan} under additional hypothesis. Let us start by showing that isometries preserve the stratification into $k-$regular sets defined by Theorem \ref{teo.rect}. 
\begin{prop}\label{prop.preservstrat}
If $g\in \Iso(X),$  and $p\in \mathcal{R}_k,$ then $gp \in \mathcal{R}_k$
\end{prop}

\begin{proof}
If $(Y,d_Y.\nu) \in \Tang(X,gp),$ it is easy to see that  $(Y,d_Y)$ is isometric to $(\mathbb{R}^k, d_E).$ Hence $Y$ has lines. Recall that $(Y,d_Y,\nu)$ is a $RCD^*(0,N)$ space, so by the Splitting Theorem \ref{teo.split}  we get that  $(Y,d_Y.\nu)$ is isomorphic as a metric measure space to $(\mathcal{R}^k, d_E,\mathcal{L}^k).$        
\end{proof}

\begin{rem}
From Theorem \ref{teo.rect}  we have that for every $k \in [1,N]\cap \mathbb{N}$ there exists a set, $\tilde{\mathcal{R}_k}\subset\mathcal{R}_k$ such that $\mathfrak{m}(\mathcal{R}_k-\tilde{\mathcal{R}_k})=0$ and $\dim_{\mathcal{H}}\tilde{\mathcal{R}_k} \leq k.$ 

In general it is not clear that  $\dim_{\mathcal{H}} \mathcal{R}_k \leq k$ or that  $G\cdot p \subset \tilde{\mathcal{R}_k}.$ So the only conclusion we can derive from the previous result is that  $\dim_T G\cdot p \leq \dim_T \mathcal{R}_k.$
\end{rem}

\begin{prop}\label{prop.dimorbit}
Let $G \leq \Iso(X)$ be a closed subgroup of positive dimension. Then for $\mathfrak{m}-$a.e. $x \in X$ the topological dimension of the orbit $G\cdot x$ is  positive.
\end{prop}
 
 \begin{proof}
 Suppose that $A:= \lbrace y\in X |\, \dim_T G\cdot y = 0 \rbrace$ has positive measure. Denote the isotropy group at point $y\in X$ as $G_y.$  Then for every point $y \in A$ we have that $\dim G = \dim G_y.$  In particular  the connected components at the identity are equal, $G_0 = (G_y)_0$, so  $G_0 \cdot y =\lbrace y \rbrace$ for all $y \in A,$ but this implies that $\Fix G_0$ has positive measure, contradicting Proposition \ref{prop.fix}.
 \end{proof}

In \cite{GalKellMonSosa}, an additional condition on the group action was imposed in order to study the $k-$regular sets on $X/G$ (See Theorem $6.3$ \cite{GalKellMonSosa}).  

\begin{defi}\label{def.locLipcoLip}
Let $G \leq \Iso(X)$ be a compact subgroup equipped  with a bi-invariant metric $d_G.$ We will say that $G$ acts \emph{locally Lipschitz and co-Lipshitz on principal orbits} if for every $y\in X$ in a principal orbit there exist constants   $C,R>0$ such that for all $0<r<R$
$$ B_y (C^{-1}r,X)\cap G\cdot y  \subset \lbrace g\cdot y \,|\,  g \in  B_{e}(r,G)\rbrace  \subset  B_{y}(Cr,X)\cap G\cdot y.  $$
\end{defi}

Recall that an $RCD^*(K,N)$  $(X,d,\mathfrak{m})$ has analytic dimension $\underline{\dim}X =k$ 
if $k$ is the largest number such that $\mathfrak{m}(\mathcal{R}_k)>0.$

\begin{prop}\label{prop.menordim}
Let $G \leq \Iso(X)$ a compact subgroup acting locally Lipschitz and co-Lipschitz on principal orbits. If the dimension of $G$ is positive then  $$\underline{\dim }X/G < \underline{\dim}X. $$
\end{prop}

\begin{proof}
 Let $x_0$ be a regular point with principal orbit $G\cdot x_0.$ Such point exists since by Theorem \ref{teo.prinorbit} and Remark \ref{rem.prinorbit} the set of points with principal orbit has full $\mathfrak{m}-$measure.  Denote by $n$ the dimension of the Euclidean space in $\Tang(X,x_0),$ and by $n^*$  the dimension of the Euclidean space in $\Tang(X/G,G \cdot x_0).$

The group $G$ satisfies the hypothesis of Theorem $6.3$ of \cite{GalKellMonSosa},so we have

$$n= \dim_T G\cdot x_0+n^* $$
Also since $G$ is of positive dimension we may assume that $\dim_T G\cdot x_0 >0,$ therefore

$$n^* < n \leq \underline{\dim} X, $$ which gives us the result. 
\end{proof}

We now extend Proposition $4.1$ of \cite{Harvey} to $RCD^*(K,N)$ spaces. The main tool we need is the lower semicontinuity under Gromov-Hausdorff convergence of the analytic dimension.   

\begin{teo}\label{teo.Lieinyectivo}
Let $G \leq \Iso(X)$ be  a compact subgroup acting locally Lipschitz and co-Lipschitz.  Take $p \in \mathcal{R}_k,$ $\lambda_n \rightarrow \infty,$ and a closed subgroup $H \leq \Iso(\mathbb{R}^k)$ such that 
$$(\lambda_nX,d_n, p, G )\rightarrow (\mathbb{R}^k,d_E,0,H), $$ in the equivariant Gromov-Hausdorff topology. Then the functions $\varphi_n: G \rightarrow H$ used in the convergence  may be chosen to be injective Lie group homomorphisms.
 
\end{teo}

\begin{proof}
Let $p\in \mathcal{R}_k,$ $\lambda_n \rightarrow \infty.$ Then we have that 
$$(\lambda_nX,d_n,\mathfrak{m}_n,p)\rightarrow (\mathbb{R}^k,d_E,\mathcal{L}^k,0) $$ in the pointed measured Gromov-Hausdorff topology. By Proposition $3.6$ of \cite{FukYam2} there exists a closed group  $H \leq \Iso(\mathbb{R}^k)$ such that
$$(\lambda_nX,d_n,p,G)\rightarrow (\mathbb{R}^k,d_E,0,H) $$ in the equivariant Gromov-Hausdorff topology. By Theorem $3.1$ of \cite{Harvey} the functions $\varphi_n: G \rightarrow H$ used in the convergence may be chosen to be Lie group homomorphisms.

Now we consider the subgroups $\Ker \varphi_n \leq G,$ which are compact.  Then by Theorem \ref{teo.quotient} the spaces $X/\Ker \varphi_n$ are still $RCD^*(K,N)$ and Proposition \ref{prop.menordim} implies that $\underline{\dim}(X/\Ker \varphi_n) \leq \underline{\dim}X =k.$

We have that  the sequence $(\lambda_n X/\Ker \varphi_n, d_n^*,\mathfrak{m}_n^*,p )$ convergences  in 
the pointed measured Gromov-Hausdorff topology to $(\mathbb{R}^k,d_E,\mathcal{L}^k,0).$ Recall that by Theorem \ref{thm.lscdim} we have

$$k= \underline{\dim}(\mathbb{R}^k,d_E, \mathcal{L}^k,0) \leq \liminf_{n \rightarrow \infty}\underline{\dim}(\lambda_n X/\Ker \varphi_n, d_n^*,\mathfrak{m}_n^*,p ), $$which implies that for large enough $n,$  $\Ker \varphi_n$ must be discrete.  So the corresponding $\varphi_n$ must  be injective. 
\end{proof}

As mentioned before, in \cite{GalKellMonSosa} an additional condition (see Definition \ref{def.locLipcoLip} ) was assumed in order to obtain information on the dimension of the tangents spaces of  $X/G.$ Here we define a weaker assumption which will be suficient for us.

\begin{defi}\label{condition.isotropy}
Let $(X,d,\mathfrak{m})$ be an $RCD^*(K,N)$ space of analytic dimension $\underline{\dim}X=k.$
A closed subgroup $G\leq \Iso(X)$ satisifes the \emph{isotropy condition} $\mathcal{I}$ if for some fixed point $p\in \mathcal{R}_k$ we have:
\begin{itemize}
\item[$(\mathcal{I})\quad$] $G_p$ acts locally Lipschitz and co-Lipschitz. 
\end{itemize}
\end{defi}

 For the remainder of this subsection we will consider closed subgroups $G$ that satisfy this condition. We need to impose this  in order to use Theorem \ref{teo.Lieinyectivo}.

\begin{prop}\label{prop.boundsisotropy}
Let $(X,d,\mathfrak{m})$ be an $RCD^*(K,N)$ space of $\underline{\dim}X = k,$ consider a closed subgroup $G \leq \Iso(X)$ satisfying the isotropy condition $\mathcal{I},$ then:

\begin{enumerate}
\item $\dim G_p \leq \frac{1}{2}k(k-1), $

\item $\dim G_p\cap G_q \leq \frac{1}{2}(k-1)(k-2),$ for $\mathfrak{m}-$a.e. $q \in X.$ 

\end{enumerate}

\end{prop}

\begin{proof}[Proof of $(1)$]
Consider  $\lambda_n \rightarrow \infty,$ From Proposition 3.6 of \cite{FukYam2}  there exists a subgroup $H\leq \Iso(\mathbb{R}^k)$ such that

$$(\lambda_nX,d_n,p,G_p) \rightarrow (\mathbb{R}^k,d_E,0,H) $$ in the equivariant Gromov-Hausdorff convergence. It is easy to check that the orbit $H\cdot 0 = \lbrace 0 \rbrace,$ hence $H \leq O(k).$ From Theorem \ref{teo.Lieinyectivo} we may choose $\varphi_n: G_p \rightarrow H$ to be injective. We then get the desired bound.
\end{proof}

\begin{proof}[Proof of $(2)$]
 Consider $q \in \mathcal{R}_k$ such that there exists a unique geodesic $\gamma,$ such that $\gamma_0=p$ and $\gamma_1=q.$ Notice that $\gamma$ must be fixed by $G_p\cap G_q.$

As in the previous item, for a sequence $\lambda_n \rightarrow \infty$ we have a group $H\leq \Iso(X)$ and the  convergence $(\lambda_nX,d_n,p,G_p\cap G_q) \rightarrow (\mathbb{R}^k,d_E, 0, H) $ in the equivariant sense.

Take a sequence  $t_n \in (0,1) $ such that $t_n \rightarrow 0$ and set $d(p, \gamma_{t_n})= \lambda_n^{-1}.$ Let $f_{n}$  be an $\epsilon_n-$approximation, we have that  $|1-d_E(0,f_n(\gamma_{t_{n}} ))|<\epsilon_n.$ So then the points $f_n(\gamma_{t_{n}} )$  converge to a point $x_0 \in \mathbb{S}^{k-1}.$

We will now prove that $H$ fixes $x_0.$ Let $h\in H,$ from the definition of equivariant convergence we have  functions $\psi_n: H \rightarrow G_p\cap G_q$ such that $$d_E(f_n(\psi_n(h)\gamma_{t_n} ),hf_n(\gamma_{t_n}))< \epsilon_n.$$

So then 
\begin{align*}
d_E(hx_0,x_0) &\leq d_E(hx_0,hf_n(\gamma_{t_n}) )+ d_E(hf_n(\gamma_{t_n}),x_0) \\
&\leq d_E(x_0,f_n(\gamma_{t_n}))+ d_E(hf_n( \gamma_{t_n}),f_{n}(\psi_n (h)\gamma_{t_n}))\\	
&\quad +d_E(f_n(\psi_n (h)\gamma_{t_n}), x_0) \\
                     &< 2d_E(x_0,f_{n}(\gamma_{t_n}))+\epsilon_n.  					  
\end{align*}
Which implies that $hx_0=x_0.$ We conclude then that $\varphi_{n}(G_p\cap G_q)$ must be isomorphic to a subgroup of $O(k-1).$ 
\end{proof}

\begin{teo}\label{teo.cotadim}
Let $G \leq \Iso(X)$ be a closed  subgroup satisfying the isotropy condition  $\mathcal{I}$, then  $$\dim G \leq \frac{1}{2}\kappa(\kappa+1), $$where $\kappa = \max\lbrace \dim_{T}\mathcal{R}_k, \underline{\dim}X\rbrace.$
\end{teo}

\begin{proof}
Let $k = \underline{\dim}X,$ and $p \in \mathcal{R}_k.$ As noted before, $G$ acts on $\mathcal{R}_k$ and so we have that $\dim_T G\cdot p  \leq \kappa.$  
The isotropy condition $\mathcal{I}$ is satisfied by $G_p$ so by Proposition \ref{prop.boundsisotropy}
 we conclude:
 
 $$\dim G = \dim G\cdot p +\dim_T G_p \leq \kappa + \frac{1}{2}k(k-1) \leq \frac{1}{2}\kappa(\kappa+1). $$
\end{proof}

\subsection{Homogeneous $RCD^*(K,N)$ spaces.}

We will now talk about transitive group actions on $RCD^*(K,N)$ spaces. If we have a compact group $G$ acting on $X$  then,  since the quotient measure $\mathfrak{m}^*$ has no atoms it is clear that orbits have  zero $\mathfrak{m}-$measure.  For general closed  groups this is still true as is shown in the next proposition. 

\begin{prop}\label{prop.orbposmeas}
Let $(X,d,\mathfrak{m})$ be an $RCD^*(K,N)$ space and $G\leq \Iso(X)$ a closed subgroup of isometries. If there exists an orbit $G\cdot p$ of positive $\mathfrak{m}-$measure then $G\cdot p = X.$
\end{prop}
Before giving the proof we recall Lemma $1.2$ of \cite{GigPas}:

\begin{lem}\label{lem.GigPas}
Let $(X,d,\mathfrak{m})$ be a doubling m.m.s. Let $E \in \mathcal{B}(X)$ and let $p\in X$ be an $\mathfrak{m}-$density point of $E.$ Then for all $\epsilon >0$  there exist $r>0$ such that for all 
$x \in B_{p}(r,X)$ there exists a point $y= y(x) \in E$ such that  $d(x,y) < \epsilon d(x,p).$
\end{lem}

\begin{proof}
We proceed by contradiction. First of all notice that if $\mathfrak{m}(G\cdot p)>0$ then $G\cdot p \subset \mathcal{R}_k$ for some $k \in [1,N]\cap \mathbb{N}.$ Without loss of generality we may assume that $p$ is a $\mathfrak{m}-$density point of its orbit.

By  Remark \ref{rem.prinorbit} of the principal orbit theorem we can take $y \in X$ such that there exists a unique $g_y p \in G\cdot p $ such that  $d(g_y p, y )=d^*(G\cdot p, G\cdot y).$ Furthermore these two points are joined by a unique geodesic $\gamma.$

Let $0<\epsilon<1,$ and $t \in (0,1]$ such that $d(p, g_y^{-1}\gamma_t)<r,$ where $r$ is the one given by Lemma \ref{lem.GigPas}.  By this same Lemma we can find a point $x \in B_p(r,X)\cap G\cdot p$ such that  $d(x, g_y^{-1}\gamma_t)< \epsilon d(p, g_y^{-1}\gamma_t)$ Therefore  $$d(x,g_y^{-1}y)< d(p,g_y^{-1}y)=d^*(G\cdot p,G\cdot y), $$ and we have a contradiction.
\end{proof}

We have this immediate corollary.

\begin{cor}
Suppose that a closed subgroup $G\leq \Iso(X)$ acts transitively on some strata $\mathcal{R}_k$ such that $\mathfrak{m}(\mathcal{R}_k)>0.$ Then $\mathcal{R}_k=X.$ 
\end{cor}

\begin{prop}\label{prop.homogeneous}
Let $(X,d,\mathfrak{m})$ be a homogeneous $RCD^*(K,N)$ space. Then $(X,d)$ is isometric to a Riemannian manifold and $\underline{\dim}X= \dim_T X.$ 
\end{prop}

\begin{proof}
Let $k=\underline{\dim}X.$   The fact that $(X,d)$ is homogeneous implies that it is locally compact, locally contractible and the metric  $d$ is intrinsic. From Theorem $3$ of \cite{Ber2} there exists a connected Lie group $G',$ and  a compact subgroup $H'\leq G'$ such that $(X,d)$ is isometric to $(G'/H',d_{cc})$ where $d_{cc}$ is a Carnot-Caratheodory-Finsler metric.

Observe that all points in $(X,d,\mathfrak{m})$ are $k-$regular. This implies that all the tangent spaces of $G'/H'$ are Euclidean spaces, so then $(G'/H',d_{cc})$ must be a Finsler manifold.

Let $g: X \rightarrow G'/H'$ be an isometry. Since $(X,d,\mathfrak{m})$ is infinitesimally Hilbertian 
then by Proposition \ref{prop.isominfHilb} $(G'/H',d_{cc},g_{\#}\mathfrak{m})$ must also be infinitesimally Hilbertian. With this we get that $(G'/H',d_{cc})$ must actually be a Riemannian manifold. Finally notice that this implies that $\underline{\dim}X=\dim_{T}X.$ 
\end{proof}

\section{Dimension Gaps. }
In this section we will prove the main results of the paper, namely the dimensional gaps that occur for closed subgroups of $\Iso(X).$ For readability we divide the proof of the first of these results into several Lemmas. First we recall the statement of the theorem:

\begin{thmp}\label{teo:B}
Let $(X,d,\mathfrak{m})$ be an $RCD^*(K,N)$ space,  such that $\underline{\dim}X =k \neq 4.$ Then there are no closed subgroups  $G \leq \Iso(X),$ satisfying the isotropy condition $\mathcal{I}$, and whose dimension lies in the interval:

$$\frac{1}{2}\kappa (\kappa-1)+1    <   \dim G <\frac{1}{2}\kappa(\kappa+1)  ,$$where  $\kappa = \max\lbrace \dim_{T}\mathcal{R}_k, \underline{\dim}X\rbrace. $

\end{thmp}

We will proceed by contradiction. Let $p\in \mathcal{R}_k$ be the point mentioned in the isotropy condition $\mathcal{I}$ (see Definition \ref{condition.isotropy}).

\begin{lem}
 $\dim G_p = \frac{1}{2}k(k-1)$
\end{lem}

\begin{proof}
We have the following

$$\frac{1}{2}\kappa(\kappa-1)+1 < \dim G= \dim G\cdot p +\dim G_p \leq \kappa+\dim G_p $$Therefore 

$$\dim G_p > \frac{1}{2}(\kappa-1)(\kappa-2)+1. $$ The group $G_p$ satisfies the hypothesis of Theorem \ref{teo.Lieinyectivo}, so we can find an injective Lie group homomorphism $\varphi: G_p \rightarrow O(k).$

This implies that $\varphi (G_p)$ is a subgroup of $O(k)$ whose  dimension is strictly larger than $\frac{1}{2}(k-1)(k-2)+1.$ Since $k\neq 4,$   Lemma of \cite[p. 48]{Koba} it shows that $\varphi(G_p)= O(k)$ or $\varphi(G_p)= SO(k).$ In either case the dimension of $G_p$ equals $\frac{1}{2}k(k-1).$  
\end{proof}

\begin{lem}
$\underline{\dim}X/G_p \leq 1,$ and $\mathfrak{m}(X-\mathcal{R}_k)=0.$
\end{lem}

\begin{proof}
Let $q \in \mathcal{R}_k$ such that  there exists a unique geodesic $\gamma$ between them and the orbit  $G_p\cdot q$ is  principal.By Proposition \ref{prop.boundsisotropy} we obtain the upper bound $\dim \left(G_p \cap G_q\right) \leq \frac{1}{2}(k-1)(k-2)$ so we have
\begin{align*}
\frac{1}{2}k(k-1)=\dim G_p &= \dim G_p\cdot q+\dim G_p\cap G_q \\
										 &\leq \dim G_p\cdot q + \frac{1}{2}(k-1)(k-2). 
\end{align*}
 Which implies that  $\dim G_p\cdot q \geq k-1.$ Since $G_p$ is compact and acts by measure preserving isometries then $(X/G_p, d^*,m^*)$ is still an $RCD^*(K,N)$ space. Let us denote by $n^*$ the dimension of the Euclidean space contained in $\Tang(X/G_p, G_p \cdot q).$ By Theorem $6.3$ of \cite{GalKellMonSosa} and the previous Lemma 
 
 $$k = \dim G_p\cdot q+n^* \geq k-1 + n^*.$$ Hence $n^* \leq 1.$ This implies that $G_p\cdot q$ is a $1-$regular point in $X/G_p.$ Using Theorem \ref{teo.lowdim} we deduce that $X/G_p$ must be isometric to a circle, to $\mathbb{R}$ or to an interval (possibly with boundary).  Observe that $p$ is fixed by $G_p$ and therefore it belongs to the boundary of the quotient space, this implies that $X/G_p$ must be isometric to $[0,a)$ $a \in \mathbb{R}^{+}\cup \lbrace \infty \rbrace$ or to $[0,b]$ $b \in \mathbb{R}^{+}. $
In either case $\underline{\dim}X/G_p =1.$ 

Now consider $l<k$ such that $\mathfrak{m}(\mathcal{R}_l)>0.$ Take a point $y \in \mathcal{R}_l$ such that $G_p\cdot y$ is principal. Then by Theorem $6.3$ of \cite{GalKellMonSosa}   $\dim G_p\cdot y = l-1.$ On the other hand we have seen that $\dim G_p\cdot y \geq k-1,$ so it must follow that $l\geq k$ and we have a contradiction. 
\end{proof}

\begin{lem}
Let $q \in X.$  Then  $G_p$  acts transitively on $\partial B_p(d(p,q),X ).$
\end{lem}

\begin{proof}
By the previous Lemma we know that $X/G_p$ is an interval.  Observe that $$d(p,u) = d^*(G_p\cdot p, G_p\cdot u)\quad \text{ for all } u \in \partial B_p (d(p,q)).$$ It is clear that $d^*(G_p\cdot q,G_p\cdot u) =0$ for all points $u$ in the boundary of the ball. So then the action  of $G_p $ must be transitive.
\end{proof}

\begin{lem}
The action of $G$ on X is transitive.
\end{lem}

\begin{proof}
We will prove that $G\cdot p $ has non-empty interior. First observe that 

$$\dim G\cdot p = \dim G -\dim G_p > \frac{1}{2}k(k-1)+1-\frac{1}{2}k(k-1) =1. $$
So we can consider a curve $\sigma : [0,1]\rightarrow X$ such that $\sigma (0)=p$ and $\sigma (t) \in G\cdot p$ for all $t \in [0,1].$

We define the function $\Lambda : [0,1] \rightarrow \mathbb{R}^{+},$  $\Lambda (t)=d(p,\sigma (t)).$ This function is continuous and achieves a maximum at say $R>0.$  Then for all $r \in (0,R)$ there exists $t_r \in [0,1]$  such that $d(p,\sigma (t_r))=r.$  It follows then that $\partial B_p(r,X)\subset G\cdot p$ and then $B_p(R,X)\subset G\cdot p.$  Since $G$ acts transitively on $G\cdot p$ then the orbit is open.
Now $X$ is connected and so we conclude that $G\cdot p =X.$
\end{proof}

From the last Lemma it follows that $X = G\cdot p$ and then $(X,d,\mathfrak{m})$ is homogeneous. Therefore using Proposition \ref{prop.homogeneous} we get that $k= \kappa.$ So then

$$\dim G =\dim_T X +\dim G_p = \frac{1}{2}k(k+1), $$ and we obtain a contradiction. We have then proved theorem $B.$
 
\begin{thmp}\label{teo:C}
Let $(X,d,\mathfrak{m})$ be an $RCD^*(K,N)$ space,  such that $\underline{\dim}X =k \neq 4.$ If  $\Iso(X)$ has dimension $\frac{1}{2}\kappa(\kappa-1)+1,$ where  $\kappa = \max\lbrace \dim_{T}\mathcal{R}_k, \underline{\dim}X\rbrace, $    and satisfies the isotropy condition $\mathcal{I}$ then $(X,d)$ is isometric to a Riemannian manifold.

\end{thmp}

\begin{proof}
By Proposition \ref{prop.homogeneous} it suffices to check that the action of the isometry group is transitive. Notice that $k\leq \kappa$ so then

$$\dim \Iso(X)_p \geq \frac{1}{2}\kappa(\kappa-1)+1 \geq \frac{1}{2}(k-1)(k-2)+1. $$

So we proceed exactly as in the proof of Theorem $B$ and obtain that the action is transitive.
\end{proof}

We also have an extension of Mann's gap theorem \cite{Mann}.

\begin{teo}
Let $(X,d,\mathfrak{m})$ be an $RCD^*(K,N)$ space of dimension $\underline{\dim}X=k$ and such that $\kappa: = \max\lbrace k, \dim_T \mathcal{R}_k \rbrace\neq 4,6,10.$ Then the isometry group $\Iso(X)$ has no compact subgroup $G$ such that the dimension of $G$ falls in the ranges:
$$ {{\kappa-l+1}\choose{2}} + {{l+1}\choose{2}}  <\dim G< {{\kappa-l+2}\choose{2}}, \quad l\geq 1.   $$

\end{teo}

\begin{proof}
From Remark \ref{rem.prinorbit} we have that $\mathfrak{m}-$a.e. point $x \in X$ is in a principal orbit. It is also clear that orbits are topological manifolds. Therefore we can follow Mann's proof (see Theorem 1 and 2 in \cite{Mann})  verbatim.  
\end{proof}

\subsection{Four dimensional $RCD^*(K,N)$ spaces.}

Now we look for the $4-$dimensional  analogues of theorems $B$ and $C.$ The first inconvenient is that  the Lemma  \cite[p. 48]{Koba}  is no longer valid. However, we have the following result:

\begin{teo}
Let $(X,d,\mathfrak{m})$ be an $RCD^*(K,N)$ space such that $\underline{\dim}X=4,$ and $\dim_T\mathcal{R}_4\leq 4$. If $G\leq \Iso(X)$ is a closed subgroup, satisfying the isotropy condition $\mathcal{I}$ and  has dimension $7$ or $8$ then $(X,d, \mathfrak{m})$ is homogeneous.
\end{teo}

\begin{proof}
Let $p \in \mathcal{R}_4$ be a point where $G$ satisfies the isotropy condition. Then we have 
$$\dim G = \dim_T G\cdot p+ \dim G_p \leq 4 +\dim G_p.  $$

Since $p\in \mathcal{R}_4,$ by  Theorem \ref{teo.Lieinyectivo}  there exists an injective Lie group homomorphism $\varphi: G_p \rightarrow SO(4).$ Observe that $SO(4)$ doesn't admit subgroups of dimension $5.$ (see  Lemma in \cite{Ishi}, p. 347). Therefore  $\dim G_p$   is either $3, 4$ or $6.$
 
Suppose $\dim G_p \geq 4.$ Since  $\varphi (G_p)$ acts on $\mathbb{S}^{3}$ by isometries and $\mathbb{S}^{3}$  is a three dimensional Alexandrov space  then  by Theorem $B$   $\dim \varphi(G_p) =6= \dim \Iso(\mathbb{S}^3).$  

From here we proceed as in the proof of theorem $B$ and conclude that $X$ is homogeneous.  

Now, if $\dim G_p =3$ then the orbit $G\cdot p$ has dimension $4.$  Since the group $G$ is closed then it acts properly on $X.$  By Theorem $2.1.4$ of \cite{Pal} we can consider a slice $S$ at $p$ that has the following properties:

\begin{itemize}
\item $G\cdot S$ is open in $X.$

\item  $S$ is closed in $X$ and $G_p-$invariant.

\item If $gS \cap S \neq \emptyset$ then  $g \in G_p.$ 
\end{itemize}  

The dimension of the isotropy group $G_p$ is positive so by  Proposition \ref{prop.dimorbit} $\mathfrak{m}-$a.e. point has an orbit of positive dimension.   So we can assume that for a point $q \in S,$  $\dim_TG_p \cdot q >0  $  and $G_p \cdot q \subset S.$

Consider $V \subset G/G_p$  a compact set homeomorphic to a closed ball  in $\mathbb{R}^4.$ Then we can find a local section $s_V: V \rightarrow G$ such that   the function $v \mapsto  v\cdot p $ is an homeomorphism. 

By intersecting $S$ with a closed ball centered at $p$ we may assume that $S$ is compact. We define the function $\Psi: V\times S \rightarrow X,$  $(v,s) \mapsto  v\cdot s.$ We see that it is an embedding.  It suffices to show that $\Psi$ is injective.   Let  $v\cdot s_1 = w\cdot s_2,$ for some $v,w \in V, s_1,s_2 \in S.$  Then $s_1 = v^{-1}w\cdot s_2$ and since $S$ is a slice this implies that $v^{-1}w \in G_p.$ But then 
$v\cdot p = w\cdot p,$  and we have that $v=w.$ It follows that $s_1 = s_2.$  

Now if we restrict  $\Psi $ to  $V \times  G_p \cdot q$ then we have that $\Psi (V \times  G_p \cdot q) \subset \mathcal{R}_4$ and  that $\dim_T \Psi (V \times  G_p \cdot q) \geq 5. $  So we have a contradiction. 
\end{proof}

\begin{rem}
Notice that from the discussion made in the Preliminaries we need to add the hypothesis on the bound of the topological dimension of $\mathcal{R}_4.$ Otherwise it is not clear that the action is transitive. 
\end{rem}

Now, using the paper of Ishihara \cite{Ishi},  we get the  $4-$dimensional analogue of Theorem $B.$ In the case that the dimension of the group $G$ is $8$ we obtain that $X$ is isometric to a K\"ahler manifold with constant holomorphic sectional curvature. (see Section $4$ \cite{Ishi}).   


\section{Symmetric spaces}

In this final section we look at symmetric spaces (compare with \cite{Ber1} and \cite{GalGui}). We will say that an $RCD^*(K,N)$ space $(X,d,\mathfrak{m})$ is locally (uniformly) symmetric if for every $p\in X$ there exists a neighbourhood $p\in U(p)$ and $r>0$ such that for all $q\in U(p)$ the ball $B_q (r,X)$ admits an isometric involution that only fixes $q.$ The space will be said to be symmetric if the involutions extend to all of $X.$  

In this final section we will see that, much like in the case of Alexandrov spaces (see Theorem $8.4$ in \cite{GalGui}), a space being symmetric is quite restrictive.  Notice that we don't need to assume anything else on the way involutions act on the reference measure $\mathfrak{m}.$

\begin{teo}
 A symmetric $RCD^*(K,N)$ space  is isometric to a Riemannian manifold. 
\end{teo}

\begin{proof}
Take a point  $x \in X$ and consider its orbit $G \cdot x.$ Consider $U(x)$ the neighbourhood of $x$ mentioned in the definition of symmetric space. We may assume this neighbourhood is a ball. 
Since $\mathfrak{m}(U(x))>0$ by  Remark \ref{rem.prinorbit}  of the Principal Orbit Theorem we have that   for $\mathfrak{m}-$a.e. $y \in X$  there exists  $\bar{x} \in G \cdot x \cap U(x)$ such that:
\begin{itemize}
\item  $y$ and $\bar{x}$ are joined by a unique geodesic, and $d(x,y) = d^*(G\cdot x, G \cdot y).$

\item  $G_y \leq G_{\bar{x}}$

\end{itemize}
 
We may further assume that there exists an involution  $\sigma_y \in G_y.$  It is clear that  $\sigma_y \in G_{\bar{x}},$  but $\sigma_y$ has only one fixed point in $U(x)$. Hence $y=\bar{x},$ and then $d^*(G\cdot x, G \cdot y)=0.$ 

With this we have that $U(x)\subset G\cdot x$  and $\mathfrak{m}(G\cdot x) >0.$  Proposition \ref{prop.orbposmeas} shows that $X=G\cdot x.$ Finally Proposition \ref{prop.homogeneous}
gives us that $(X,d,\mathfrak{m})$ must be a Riemannian manifold.
\end{proof}

In \cite{Ber1} Berestovski\u{\i} studied locally symmetric spaces as well and he showed that in simply connected $G-$spaces the two notions coincide. However, examples $8.1, 8.2$ and $8.3$ constructed in \cite{GalGui} exhibit that this is no longer true for Alexandrov spaces.

\end{document}